\documentclass{article}

\usepackage[margin=1.2in]{geometry}
\usepackage{amsmath,amssymb,amsfonts,graphicx,amsthm,nicefrac,mathtools,bm}
\usepackage{hyperref}
\usepackage[numbers]{natbib}
\usepackage[english]{babel}
\usepackage{times}
\usepackage[T1]{fontenc}
\usepackage{bbm}
\usepackage{color}
\usepackage{xspace}
\usepackage{rotating,multirow}
\usepackage{authblk}
\newtheorem{theorem}{Theorem}

\newtheorem{lemma}{Lemma}

\theoremstyle{definition}
\newtheorem{definition}{Definition}
\theoremstyle{remark}

\makeatletter
\newtheorem*{rep@theorem}{\rep@title}
\newcommand{\newreptheorem}[2]{%
\newenvironment{rep#1}[1]{%
 \def\rep@title{#2 \ref{##1}}%
 \begin{rep@theorem}}%
 {\end{rep@theorem}}}
\makeatother
\newreptheorem{theorem}{Theorem}
\newreptheorem{lemma}{Lemma}
\newreptheorem{corollary}{Corollary}
\newreptheorem{proposition}{Proposition}

\newcommand{\figref}[1]{Figure~\ref{fig:#1}}

\newcommand{\secref}[1]{Section~\ref{sec:#1}}
\newcommand{\appref}[1]{Appendix~\ref{app:#1}}

\newcommand{\lemref}[1]{Lemma~\ref{lem:#1}}

\newcommand{\thmref}[1]{Theorem~\ref{thm:#1}}

\newcommand{\eqnref}[1]{\eqref{eqn:#1}}

\DeclareMathOperator*{\argmin}{arg\,min}

\newcommand{\iid}[0]{i.i.d.\xspace}

\newcommand{\One}[1]{{\mathbbm{1}}\left\{{#1}\right\}}
\newcommand{\inner}[2]{\langle{#1},{#2}\rangle} 
\newcommand{\norm}[1]{\lVert{#1}\rVert}
\newcommand{\Norm}[1]{\left\lVert{#1}\right\rVert}
\newcommand{\PP}[1]{\mathbb{P}\left\{{#1}\right\}} 
\newcommand{\EE}[1]{\mathbb{E}\left[{#1}\right]} 

\def\R{\mathbb{R}}

\newcommand{\ee}[1]{\mathbf{e}_{{#1}}}
\newcommand{\ident}{\mathbf{I}}
\newcommand{\ones}{\mathbf{1}}

\newcommand{\iidsim}{\stackrel{\mathrm{iid}}{\sim}}

\newcommand{\ignore}[1]{}

\newcommand{\pr}[1]{\mathcal{P}_{{#1}}}

\newcommand{\eps}{\epsilon}

\usepackage{fancyhdr}
\newcommand{\thedate}{\today}
\newcommand{\theauthor}{Fan Yang and Rina Foygel Barber}
\newcommand{\thetitle}{Contraction and uniform convergence of isotonic regression}

\date{\thedate}
\author{\theauthor}
\title{\thetitle}
\affil{Department of Statistics, University of Chicago}

\fancypagestyle{plain}{%
  \fancyhf{}%
  \lhead{\thepage}
}

\usepackage[mmddyy]{datetime}

\newcommand{\iso}{\mathsf{iso}}

\newcommand{\sw}[1]{\Norm{#1}^{\mathsf{SW}}_{\psi}}
\newcommand{\Lip}{{L}}

\newcommand{\eiso}{\eps_{\iso}}

\newcommand{\cmin}{{c}}

\newcommand{\xh}{\widehat{x}}

\newcommand{\fh}{\widehat{f}}
\newcommand{\Gh}{\widehat{G}_n}
\newcommand{\Ggren}{\widehat{G}_{\mathsf{Gren}}}
\newcommand{\ggren}{\widehat{g}_{\mathsf{Gren}}}
\newcommand{\ck}{\mathcal{K}}

\begin{document}

\maketitle

\begin{abstract}
\, We consider the problem of isotonic regression, where the underlying signal $x$ is assumed to satisfy a monotonicity constraint, that is, $x$ lies in the cone $\{ x\in\R^n : x_1 \leq \dots \leq x_n\}$. We study the isotonic projection operator (projection to this cone), and find a necessary and sufficient condition characterizing all norms with respect to which this projection is contractive. This enables a simple and non-asymptotic analysis of the convergence properties of isotonic regression, yielding uniform confidence bands that adapt to the local Lipschitz properties of the signal.
\end{abstract}

\section{Introduction}

Isotonic regression is a powerful nonparametric tool used for estimating a monotone 
signal from noisy data. Specifically, our  data consists of 
observations $y_1,\dots,y_n\in\R$, which are assumed to be noisy observations of some 
monotone increasing signal---for instance, we might assume that $\EE{y_1}\leq \dots\leq \EE{y_n}$.
Isotonic (least-squares) regression solves the optimization problem
\[\textnormal{Minimize $\norm{y-x}^2_2$ subject to $x_1\leq \dots \leq x_n$}\]
in order to estimate the underlying signal. 

This regression problem can be viewed as a convex projection, since $\ck_{\iso} = \{x\in\R^n: x_1\leq\dots\leq x_n\}$
is a convex cone. We will write
\[\iso(y) \coloneqq  \pr{\ck_{\iso}}(y) = \argmin_{x\in\R^n}\big\{\norm{y-x}^2_2 : x_1\leq\dots\leq x_n\big\} \]
to denote the projection to this cone, which solves the least-squares isotonic regression problem. 
This projection can be computed in finite time with the Pool Adjacent Violators Algorithm (PAVA), developed by \citet{barlow1972statistical}.

In fact, mapping $y$ to $\iso(y)$ is known to also solve the isotonic binary regression problem. This arises when the data is binary,
that is,  $y\in\{0,1\}^n$. If we assume that $y_i \sim\textnormal{Bernoulli}(x_i)$, then the constrained maximum likelihood estimator
is exactly equal to the projection $\iso(y)$ (\citet{robertsonorder}). 

In this paper, we examine the properties of the isotonic projection operator $x\mapsto \iso(x)$,
with respect to different norms $\norm{\cdot}$ on $\R^n$. We examine the conditions on $\norm{\cdot}$ needed in order to ensure
that $x\mapsto\iso(x)$ is contractive with respect to this norm, and in particular, we define the new ``sliding window norm'' which measures
weighted averages over ``windows'' of the vector $x$, i.e.~contiguous stretches of the form $(x_i,x_{i+1},\dots,x_{j-1},x_j)$
for some indices $1\leq i\leq j\leq n$.
This sliding window norm then provides a tool for analyzing the convergence behavior of isotonic regression in a setting
where our data is given by $y_i = x_i + \text{noise}$. If the underlying signal $x$ is believed to be (approximately) monotone increasing,
then $\iso(y)$ will provide a substantially better estimate of $x$ than the observed vector $y$ itself.
By using our results on contractions with respect to the isotonic projection operator, 
we obtain clean, finite-sample bounds on the pointwise errors, $\big|x_i - \iso(y)_i\big|$, which are locally adaptive to the behavior
of the signal $x$ in the region around the index $i$ and hold uniformly over the entire sequence.

\subsection{Background}\label{sec:intro_refs}
There is extensive literature studying convergence rates of isotonic regression, in both finite-sample and asymptotic settings.
For an asymptotic formulation of the problem, since the signal $x\in\R^n$ must necessarily change as $n\rightarrow\infty$, a standard
method for framing this as a sequence of problems indexed by $n$
is to consider a fixed function $f:[0,1]\rightarrow\infty$, and then for each $n$, define
$x_i = f(i/n)$ (or more generally, $x_i = f(t_i)$ for points $t_i$ that are roughly uniformly spaced).
Most models in the literature assume that $y_i = x_i + \sigma\cdot \eps_i$, where the noise terms $\eps_i$ are \iid standard normal variables
 (or, more generally, are zero-mean variables that satisfy some moment assumptions or are subgaussian).

One class of existing results treats {\em global} convergence rates, where the goal is to bound the error $\norm{x - \iso(y)}_2$, or more generally
to bound $\norm{x-\iso(y)}_p$ for some $\ell_p$ norm. 
The estimation error under the $\ell_2$ norm was studied by \citet{van1990estimating,wang1996l2risk,meyer2000degrees}, among others. \citet{van1993hellinger} obtains the asymptotic risk bounds for certain `bounded' isotonic regression under Hellinger distance, whereas \citet{zhang2002risk}  establishes the non-asymptotic risk bounds for general $\ell_p$ norm---in particular, for $p=2$, they show that the least-squares estimator $\iso(y)$ of the signal $x$ has error scaling as $\norm{x-\xh}_2/\sqrt{n} \sim n^{-1/3}$. Recent work by \citet{chatterjee2015risk} considers non-asymptotic minimax rates for the estimation error, focusing specifically on $\norm{x-\xh}_2$
for any estimator $\xh$ to obtain a minimax rate.   Under a Gaussian noise model, they prove that the minimax rate scales as $\norm{x-\xh}_2/\sqrt{n} \gtrsim n^{-1/3}$
over the class of monotone and Lipschitz signals $x$, which matches the error rate of the constrained maximum likelihood estimator (i.e.~the isotonic
least-squares projection, $\iso(y)$) established earlier. They also study minimax rates in a range of settings, including piecewise constant signals, which we will discuss
later on.\footnote{\citet{chatterjee2015risk}'s  results, which they describe as ``local minimax''
bounds,
  are ``local'' in the sense that the risk bound they provide is specific to an individual signal $x\in\R^n$, but the error is nonetheless measured
with respect to the $\ell_2$ norm, i.e.~``globally'' over the entire length of the signal.}

A separate class of results considers {\em local} convergence rates, where the error at a particular index, i.e.~$\big|x_i - \iso(y)_i\big|$ for some particular $i$,
may scale differently in different regions of the vector. In an asymptotic setting, where $n\rightarrow\infty$ and the underlying signal comes from
a function $f:[0,1]\rightarrow\R$, we may consider an estimator $\fh:[0,1]\rightarrow\R$, where 
$\fh(t)$ is estimated via $\iso(y)_i$ for $t\approx i/n$.
Results in the literature for this setting
study the asymptotic rate of convergence of $\big|f(t) - \fh(t)\big|$, which depends on the local properties of $f$
near $t$.  \citet{brunk1969estimation} establishes the convergence rate as well as the limiting distribution when $f'(t)$ is positive, whereas \citet{wright1981asymptotic} generalizes the result to the case of $t$ lying in a flat region,
i.e.~$f'(t)=0$. \citet{
cator2011adaptivity} shows that the isotonic estimator adapts to the unknown function locally and is asymptotically minimax optimal for local behavior.
 Relatedly, \citet{dumbgen2003optimal} gives confidence bands in the related Gaussian white noise model, by taking averages over windows of the data curve, i.e.~ranges of the form $[t_0,t_1]$ near the point $t$ of interest.

In addition, many researchers have considered the related problem of 
monotone density estimation, where we aim to estimate a monotone decreasing density from $n$ samples drawn from that distribution. This problem was first studied by \citet{grenander1956theory}, and has attracted 
much attention since then, see \citet{rao1969estimation,groeneboom1984estimating,birge1987estimating,birge1993rates,carolan1999asymptotic,balabdaoui2011grenander,jankowski2014convergence}, among others.
\citet{birge1987estimating} proves a $n^{-1/3}$ minimax rate for the $\ell_2$ error in estimating the true monotone density $f(t)$---the same rate as 
for the isotonic regression problem.
The pointwise i.e.~$\ell_\infty$ error has also been studied---\citet{durot2012limit} shows that, for Lipschitz and bounded densities on $[0,1]$, asymptotically the error rate for estimating $f(t)$ scales
as $(n/\log(n))^{-1/3}$, uniformly over all $t$ bounded away from the endpoints.
Adaptive convergence rates are studied by \citet{cator2011adaptivity}.
Later we will show that our results yield a non-asymptotic error bound for this problem as well, which matches this known rate.

Several related problems for isotonic regression have also been studied. 
First, assuming the model $y_i = x_i + \sigma\cdot \eps_i$ for standard normal error terms $\eps_i$,
estimating $\sigma$ has been studied by \citet{meyer2000degrees}, among others.
 Estimators of $\sigma$ for general distribution of $\eps_i$ are also available, see \citet{rice1984bandwidth,gasser1986residual}.
 We discuss the relevance of these tools for constructing our confidence bands in \secref{bands}.
Second, we can hope that our estimator $\iso(y)$ can recover $x$ accurately only if $x$ itself
is monotone (or approximately monotone); thus, testing this hypothesis is important for knowing whether
our confidence band can be expected to cover $x$ itself or only its best monotone approximation, $\iso(x)$.
\citet{drton2010geometric} study the problem of testing the null hypothesis $x\in\ck_\iso$ (or more generally,
whether the signal $x$ belongs to some arbitrary pre-specified cone $\ck$), based on the volumes
of lower-dimensional faces of the cone (see \citet[Theorem 2 and Section 3]{drton2010geometric}).

\paragraph{Main contributions}
In the context of the existing literature, our main contributions are: (1) the new analysis of the contraction properties of isotonic projection, and
the specific example of the sliding window norm, and (2) clean, finite-sample estimation bounds for isotonic regression, which
 are locally adaptive to the local Lipschitz behavior of the 
underlying signal, and match known asymptotic convergence rates. The contraction and sliding window norm allow us to prove
the isotonic regression convergence results in just a few simple lines of calculations, while the arguments in the existing literature are generally substantially
more technical (for example, approximating the estimation process via a Brownian motion or Brownian bridge).

\section{Contractions under isotonic projection}
In this section, we examine the contractive behavior of the isotonic projection, 
\[\iso(x) = \argmin_{y\in\R^n}\big\{\norm{x-y}_2 : y_1\leq \dots \leq y_n\big\},\]
with respect to various norms on $\R^n$.
Since this operator projects $x$ onto a convex set (the cone $\ck_{\iso}$ of all ordered vectors), it is trivially true that
\[\norm{\iso(x)-\iso(y)}_2\leq \norm{x-y}_2,\]
but we may ask whether the same property holds when we consider norms other than the $\ell_2$ norm.

Formally, we defined our question as follows:
\begin{definition}\label{def:iso_contract}
For a seminorm $\norm{\cdot}$ on $\R^n$, we say that
isotonic projection is contractive with respect to $\norm{\cdot}$
if
\[\norm{\iso(x)-\iso(y)}\leq\norm{x-y}\text{ for all $x,y\in\R^n$.}\]
\end{definition}
We recall that a seminorm must satisfy a scaling law, $\norm{c\cdot x} = |c|\cdot\norm{x}$, and the triangle inequality, $\norm{x+y}\leq\norm{x}+\norm{y}$,
but may have $\norm{x}=0$ even if $x\neq 0$. From this point on, for simplicity, we will simply say ``norm'' to refer to any seminorm.

For which types of norms can we expect 
this contraction property to hold? To answer this question, we first define a simple property to help our analysis:
\begin{definition}\label{def:NUNA}
For a norm $\norm{\cdot}$ on $\R^n$,
we say that $\norm{\cdot}$ is nonincreasing under neighbor averaging (NUNA)
if
\begin{equation*}
\Norm{\Big(x_1,\dots,x_{i-1},\frac{x_i+x_{i+1}}{2},\frac{x_i+x_{i+1}}{2},x_{i+2},\dots,x_n\Big)}
\leq \norm{x}  
\end{equation*}
$\text{ for all } x\in\R^n\text{ and all }i=1,\dots,n-1.$
\end{definition}

Our first main result proves that the NUNA property exactly characterizes the contractive behavior of isotonic projection---NUNA
is both necessary and sufficient for isotonic projection to be contractive.
\begin{theorem}\label{thm:iso_contract}
For any norm $\norm{\cdot}$ on $\R^n$,
isotonic projection is contractive with respect to $\norm{\cdot}$
if and only if $\norm{\cdot}$   is nonincreasing under neighbor averaging (NUNA).
\end{theorem}
(The proof of \thmref{iso_contract} will be given in \secref{proof_nuna}.)

In particular, this theorem allows us to easily prove that isotonic projection
is contractive with respect to the $\ell_p$ norm for any $p\in[1,\infty]$, and more
generally as well, via the following lemma:
\begin{lemma}\label{lem:perms}
Suppose that $\norm{\cdot}$ is a norm that is invariant to permutations of the entries
of the vector, that is, for any $x\in\R^n$ and any permutation $\pi$ on $\{1,\dots,n\}$,
\[\norm{x} = \norm{x_\pi}\text{ where }x_\pi\coloneqq (x_{\pi(1)},\dots,x_{\pi(n)}).\]
(In particular, the $\ell_p$ norm, for any $p\in[1,\infty]$, satisfies this property.)
Then $\norm{\cdot}$ satisfies the NUNA property, and therefore isotonic projection is a contraction with respect to $\norm{\cdot}$.
\end{lemma}
\begin{proof}[Proof of \lemref{perms}]
Let $\pi$  swap indices $i$ and $i+1$, so 
that 
\[ x_{\pi} =  (x_1,\dots,x_{i-1},x_{i+1},x_i,x_{i+2},\dots,x_n).\]
Then 
\begin{multline*}
\Norm{\Big(x_1,\dots,x_{i-1},\frac{x_i+x_{i+1}}{2},\frac{x_i+x_{i+1}}{2},x_{i+2},\dots,x_n\Big)}
\\ = \Norm{\frac{x + x_\pi}{2}} \leq \frac{1}{2}\left(\norm{x} + \norm{x_{\pi}}\right) = \norm{x},
\end{multline*}
where we apply the triangle inequality, and the assumption that $\norm{x_\pi}=\norm{x}$.
This proves that $\norm{\cdot}$ satisfies NUNA. By \thmref{iso_contract}, this implies that isotonic
projection is contractive with respect to $\norm{\cdot}$.
\end{proof}

\section{The sliding window norm}
We now introduce a {\em sliding window} norm,
which will later be a useful tool for obtaining uniform convergence guarantees for isotonic regression.
For any pair of indices $1\leq i\leq j\leq n$, we write $i:j$ to denote the stretch of $j-i+1$ many coordinates
indexed by $\{i,\dots,j\}$,
\[x = (x_1,\dots,x_{i-1},\underbrace{x_i,\dots,x_j,}_{\text{window $i$:$j$}}x_{j+1},\dots,x_n).\]
Fix any function
\begin{equation}\label{eqn:psi}
\psi:\{1,\dots,n\}\rightarrow \R_+\text{ such that $\psi$ is nondecreasing and $i\mapsto i/\psi(i)$ is concave.}\end{equation}
The sliding window norm is defined as
\[\sw{x} = \max_{1\leq i\leq j\leq n}\Big\{ \big|\overline{x}_{i:j}\big| \cdot \psi(j-i+1)\Big\},\]
where $\overline{x}_{i:j} = \frac{ x_i+\dots+x_j}{j-i+1}$ denotes the average over the window $i:j$.

The following key lemma proves that our contraction theorem, \thmref{iso_contract}, can be applied to this sliding
window norm. (This lemma, and all lemmas following, will be proved in \appref{proofs_lemmas}.)
\begin{lemma}\label{lem:SW_contract}
For any function $\psi$ satisfying the conditions~\eqnref{psi}, 
the sliding window norm $\sw{\cdot}$ satisfies the NUNA property, and therefore,
isotonic projection is contractive with respect to this norm.
\end{lemma}
This lemma is a key ingredient to our convergence analyses for isotonic regression.
It will allow us to use the sliding window norm to understand the behavior of $\iso(y)$ as an estimator of $\iso(x)$,
where $y$ is a vector of noisy observations of some target signal $x$.
In particular, we will consider the special case of subgaussian noise \footnote{ We call a random variable $X$ subgaussian if $\PP{|X-\EE{X}| \geq t }\leq 2\exp(-t^2/2) $ for any $t>0$.} . The following lemma can be proved with a very
basic union bound argument:
\begin{lemma}\label{lem:SW_subg}
Let $x\in\R^n$ be a fixed vector, and let
$y_i = x_i + \sigma \eps_i$, where the $\eps_i$'s are independent, zero-mean, and subgaussian.
Then taking $\psi(i)=\sqrt{i}$, we have
\[\EE{\sw{x-y}}\leq \sqrt{2\sigma^2\log(n^2+n)}\text{ and }\EE{\big(\sw{x-y}\big)^2}\leq 8\sigma^2\log(n^2+n),\]
and for any $\delta>0$,
\[\PP{\sw{x-y}\leq\sqrt{2\sigma^2\log\left(\frac{n^2+n}{\delta}\right)}}\geq 1-\delta.\]
\end{lemma}
As a specific example, in a Bernoulli model, if the signal is given by $x\in[0,1]^n$ and our observations are 
given by $y_i\sim\textnormal{Bernoulli}(x_i)$ (each drawn independently),
then this model satisfies the subgaussian noise model with $\sigma=1$.

\section{Estimation bands}\label{sec:bands}
In this section, we will develop a range of results bounding our estimation error when we observe a (nearly) monotone signal plus noise.
These results will all
use the sliding window contraction result in \lemref{SW_contract} as the main ingredient in our analysis.

We begin with a deterministic statement that is a straightforward consequence of the sliding window contraction result:
\begin{theorem}\label{thm:backbone}
For any $x,y\in\R^n$, for all indices $k=1,\dots,n$,
\begin{equation}\label{eqn:backbone_y}
\begin{split}
& \max_{1\leq m\leq k}\left\{\overline{\iso(y)}_{(k-m+1):k} - \frac{\sw{x-y}}{\psi(m)}\right\} \\ & \leq \iso(x)_k \leq \min_{1\leq m\leq n-k+1}\left\{\overline{\iso(y)}_{k:(k+m-1)} + \frac{\sw{x-y}}{\psi(m)}\right\}
\end{split}
\end{equation}
and 
\begin{equation}\label{eqn:backbone_x}
\begin{split}
& \max_{1\leq m\leq k}\left\{\overline{\iso(x)}_{(k-m+1):k} - \frac{\sw{x-y}}{\psi(m)}\right\} \\ & \leq \iso(y)_k \leq \min_{1\leq m\leq n-k+1}\left\{\overline{\iso(x)}_{k:(k+m-1)} + \frac{\sw{x-y}}{\psi(m)}\right\}.
\end{split}
\end{equation}
\end{theorem}
Note that these two statements are symmetric; they are identical up to reversing the roles of $x$ and $y$.
\begin{proof}[Proof of \thmref{backbone}]
We have $\iso(x)_k \geq \overline{\iso(x)}_{(k-m+1):k} \geq \overline{\iso(y)}_{(k-m+1):k} - \frac{\sw{x-y}}{\psi(m)}$,
where the first inequality uses the monotonicity of $\iso(x)$ while the second uses the definition of the sliding window norm along with the fact
that $\sw{\iso(x)-\iso(y)}\leq\sw{x-y}$ by \lemref{SW_contract}. This proves the lower bound for~\eqnref{backbone_y}; the upper bound,
and the symmetric result~\eqnref{backbone_x}, are proved analogously.
\end{proof}
This simple reformulation of our contraction result, in fact forms the backbone of all our estimation band guarantees.

These bounds bound the difference between $\iso(x)$ and $\iso(y)$, 
 computed using either $y$ (as in~\eqnref{backbone_y}) or $x$ (as in~\eqnref{backbone_x}).
Thus far, the two results are entirely symmetrical---they are the same if we swap the vectors $x$ and $y$.

We will next study the statistical setting where we aim to estimate a signal $x$ based on noisy observations $y$, 
in which case the vectors $x$ and $y$ play distinct roles, and so the two versions of the bands will carry entirely different meanings.
Before proceeding, we note that the above bounds cannot give results on $x$ itself, but only on its projection $\iso(x)$.
If $x$ is far from monotonic, we cannot hope that the monotonic vector $\iso(y)$
would give a good estimate of $x$.
We will consider a relaxed monotonicity constraint: we say that $x\in\R^n$
is $\eiso$-monotone if
\[x_i \leq x_j + \eiso \text{ for all $1\leq i\leq j\leq n$.}\]
(If $x$ is monotonic then we can simply set $\eiso=0$.)
We find that $\eiso$ corresponds roughly to
 the $\ell_{\infty}$ distance between $x$ and its isotonic projection $\iso(x)$:
\begin{lemma}\label{lem:eiso}
For any $x\in\R^n$ that is $\eiso$-monotone,
\[\norm{x - \iso(x)}_{\infty} \leq \eiso.\]
Conversely, any $x\in\R^n$ with $\norm{x-\iso(x)}_{\infty}\leq \eps$ must be $(2\eps)$-monotone.
\end{lemma}

With this in place, we turn to our results for the statistical setting.

\subsection{Statistical setting}\label{sec:model}
We will consider a subgaussian noise model, where $x\in\R^n$ is a fixed signal,
and the observation vector $y$ is generated as
\begin{equation}\label{eqn:stat_model}y_i = x_i + \sigma \eps_i, \text{ where the $\eps_i$'s are independent, zero-mean, and subgaussian}.\end{equation}
\lemref{SW_subg} proves that, in this case, setting $\psi(m)=\sqrt{m}$ 
would yield $\sw{x-y}\leq \sqrt{2\sigma^2\log\left(\frac{n^2+n}{\delta}\right)}$ with probability at least $1-\delta$.
Of course, we could consider other models as well, e.g.~involving correlated noise or heavy-tailed noise, but restrict our attention
to this simple model for the sake of giving an intuitive illustration of our results.
 
In order for this bound on the sliding window to be useful in practice, we need to obtain a bound or an estimate for the noise level $\sigma$.
Under the Bernoulli model $y_i\sim\textnormal{Bernoulli}(x_i)$,
we can simply set $\sigma=1$.
More generally, it may be possible to estimate $\sigma$ from the data itself, for instance if the noise terms $\eps_i$ are \iid~standard normal, 
\citet{meyer2000degrees} propose estimating the noise level $\sigma$ with the maximum likelihood estimator (MLE),
 $\widehat{\sigma}^2 = \frac{1}{n} \sum_i (y_i - \iso(y)_i)^2$, or the bias-corrected MLE given by 
\[\widehat{\sigma}^2 = \frac{\sum_i (y_i - \iso(y)_i)^2}{n - c_1 \cdot \mathsf{df}\big(\iso(y)\big)},\]
where $c_1$ is a known constant while $\mathsf{df}\big(\iso(y)\big)$ is the number of ``degrees of freedom'' in the monotone vector $\iso(y)$,
i.e.~the number of distinct values in this vector.

We next consider the two different types of statistical guarantees that can be obtained, using the two symmetric formulations in \thmref{backbone}
above.

\subsection{Data-adaptive bands}
We first consider the problem of providing a confidence band for the signal $x$ in a practical setting, where we can only observe the noisy data $y$
and do not have access to other information.  In this setting, the bound~\eqnref{backbone_y} in \thmref{backbone}, combined with \lemref{SW_subg}'s bound on $\sw{x-y}$ for the subgaussian model,
yields the following result:
\begin{theorem}\label{thm:subg_adapt}
For any signal $x\in\R^n$ and any $\delta>0$, under the subgaussian noise model~\eqnref{stat_model}, then with probability 
at least $1-\delta$, for all $k=1,\dots, n$,
\begin{equation}\label{eqn:subg_adapt}
\begin{split}
& \max_{1\leq m\leq k}\left\{\overline{\iso(y)}_{(k-m+1):k} - \sqrt{\frac{2\sigma^2\log\left(\frac{n^2+n}{\delta}\right)}{m}}\right\}
 \\  & \leq \iso(x)_k   \leq \min_{1\leq m\leq n-k+1}\left\{\overline{\iso(y)}_{k:(k+m-1)} + \sqrt{\frac{2\sigma^2\log\left(\frac{n^2+n}{\delta}\right)}{m}}\right\}.
 \end{split}
\end{equation}
If additionally $x$ is $\eiso$-monotone, then we also have
\begin{equation}\label{eqn:subg_adapt_eiso}
\begin{split}
& \max_{1\leq m\leq k}\left\{\overline{\iso(y)}_{(k-m+1):k} - \sqrt{\frac{2\sigma^2\log\left(\frac{n^2+n}{\delta}\right)}{m}}\right\} - \eiso \\ & 
\leq x_k \leq \min_{1\leq m\leq n-k+1}\left\{\overline{\iso(y)}_{k:(k+m-1)} + \sqrt{\frac{2\sigma^2\log\left(\frac{n^2+n}{\delta}\right)}{m}}\right\} + \eiso.
\end{split}
\end{equation} 
\end{theorem}
We emphasize that these bounds give us a uniform confidence band for $\iso(x)$ (or for $x$ itself, if it is monotone) that can be computed
without assuming anything about the properties of the signal; for instance, we do not assume that the signal is Lipschitz with some known constant, or anything of this
sort. We only need to know the noise level $\sigma$, which can be estimated as discussed in \secref{model}.
In this sense, the bounds are data-adaptive---they are computed using the observed projection $\iso(y)$, and adapt to the properties of 
the signal (for instance, if $x$ is locally constant near $k$, then the upper and lower confidence bounds will be closer together).

{ 
\paragraph{Comparison to existing work}
The flavor of our data-adaptive band is close to that given in \citet{dumbgen2003optimal}, where the author gives confidence bands for signals in a continuous Gaussian white noise model. Although in Section 5 of \citet{dumbgen2003optimal} the result is applied to the discrete case, the confidence band there is only valid asymptotically as pointed our by the author, whereas our band is valid for finite samples. Moreover, the computation of the band in \citet{dumbgen2003optimal} involves Monte Carlo simulation to estimate several key quantiles,  and hence is much heavier than the computation of our band. Another difference is that \citet{dumbgen2003optimal} employes kernel estimators in their bands while we use the isotonic least squares estimator in our construction.
}

\subsection{Convergence rates}
While the results of \thmref{subg_adapt} give data-adaptive bounds that do not depend on properties of $x$,
from a theoretical point of view we would also like to understand how the estimation error depends on these properties.
For the data-adaptive bands, we used the result~\eqnref{backbone_y} relating $\iso(x)$ and $\iso(y)$, but for this question, we will use 
the symmetric result~\eqnref{backbone_x} instead, which immediately yields the following theorem.
\begin{theorem}\label{thm:subg_theory}
For any signal $x\in\R^n$ and any $\delta>0$, under the subgaussian noise model~\eqnref{stat_model}, then with probability 
at least $1-\delta$, for all $k=1,\dots, n$,
\begin{multline}\label{eqn:subg_theory}
-\min_{1\leq m\leq k}\left\{ \left(\iso(x)_k - \overline{\iso(x)}_{(k-m+1):k} \right) + \sqrt{\frac{2\sigma^2\log\left(\frac{n^2+n}{\delta}\right)}{m}}\right\}\\
\leq \iso(y)_k - \iso(x)_k \leq \min_{1\leq m\leq n-k+1}\left\{\left(\overline{\iso(x)}_{k:(k+m-1)}-\iso(x)_k\right) +   \sqrt{\frac{2\sigma^2\log\left(\frac{n^2+n}{\delta}\right)}{m}}\right\}.
\end{multline}
If additionally $x$ is $\eiso$-monotone, then we also have
\begin{multline}\label{eqn:subg_theory_eiso}
-\min_{1\leq m\leq k}\left\{ \left(x_k - \overline{x}_{(k-m+1):k} \right) + \sqrt{\frac{2\sigma^2\log\left(\frac{n^2+n}{\delta}\right)}{m}}\right\} - \eiso\\
\leq \iso(y)_k - x_k \leq \min_{1\leq m\leq n-k+1}\left\{\left(\overline{x}_{k:(k+m-1)}-x_k\right) +   \sqrt{\frac{2\sigma^2\log\left(\frac{n^2+n}{\delta}\right)}{m}}\right\} + \eiso.
\end{multline} 
\end{theorem}
\begin{proof}[Proof of \thmref{subg_theory}]
For the first bound~\eqnref{subg_theory}, we simply subtract $\iso(x)_k$ from the inequalities~\eqnref{backbone_x}.
 For the second bound~\eqnref{subg_theory_eiso} in the case that $x$ 
is approximately monotone, we instead subtract $x_k$ from~\eqnref{backbone_x},
and also use the fact that $\norm{x-\iso(x)}_{\infty}\leq\eiso$ by \lemref{eiso},
which implies that $\big|\overline{x}_{k:(k+m-1)} - \overline{\iso(x)}_{k:(k+m-1)}\big|\leq \eiso$,
and similarly $\big| \overline{x}_{(k-m+1):k} -  \overline{\iso(x)}_{(k-m+1):k}\big|\leq \eiso$.
\end{proof}

\paragraph{Comparison to existing work}
In the monotone setting (i.e.~$x=\iso(x)$), 
\citet{chatterjee2015risk} derive related results bounding the pointwise error $\big|x_k - \iso(y)_k\big|$.
Specifically, they use the ``minmax'' formulation of the isotonic projection, $\iso(y)_k =  \min_{j\geq k}\max_{i\leq k} \overline{y}_{i:j}$,
and give the following argument:
\begin{multline*}
\iso(y)_k - x_k= \min_{1\leq m\leq n-k+1} \max_{i\leq k}\overline{y}_{i:(k+m-1)} - x_k\\
\leq \min_{1\leq m\leq n-k+1} \left\{\left(\max_{i\leq k}\overline{x}_{i:(k+m-1)} - x_k\right)  + \max_{i\leq k} \big|\overline{x}_{i:(k+m-1)} - \overline{y}_{i:(k+m-1)} \big|\right\}\\
\leq\min_{1\leq m\leq n-k+1}\bigg\{\left( \overline{x}_{k:(k+m-1)}  - x_k\right)+\underbrace{
 \max_{i\leq k} \big|\overline{x}_{i:(k+m-1)} - \overline{y}_{i:(k+m-1)} \big|}_{\text{(Err)}}\bigg\},
\end{multline*}
where the first step defines $m=j-k+1$ and uses the ``minmax'' formulation,
 while the third uses the assumption that $x$ is monotone.
They then bound the error term (Err)
in expectation. We can instead bound it as $\text{(Err)}\leq \frac{\sw{x-y}}{\sqrt{m}}$, which is exactly
the same as the upper bound in our result~\eqnref{subg_theory_eiso}. Their ``minmax'' strategy
can analogously be used to obtain the corresponding lower bound as well.

\subsection{Locally constant and locally Lipschitz signals}
If the signal $x$ is monotone, \citet{chatterjee2015risk}'s results, which are analogous to our bounds in~\eqnref{subg_theory_eiso},
 yield implications for many different classes of signals: for instance,
they show that for a piecewise constant signal $x$ taking only $s$ many unique values,
the $\ell_2$ error scales as
\[\frac{1}{n}\norm{x-\iso(y)}^2_2 \leq \frac{16 s \sigma^2}{n} \log\left(\frac{en}{s}\right).\]
We therefore see that 
\begin{equation}\label{eqn:half}\big|x_k - \iso(y)_k\big| \lesssim \sqrt{\frac{log(n)}{n}}\end{equation}
for ``most'' indices $k$ when the signal is piecewise constant.

We can instead consider a Lipschitz signal: we say that $x$ is $\Lip$-Lipschitz
if $|x_i - x_{i+1}|\leq L/n$ for all $i$. (Rescaling by $n$ is natural as we often think of $x_i = f(i/n)$ 
for some underlying function $f$). In this setting, our results in \thmref{subg_theory}
immediately yield the bound
\begin{equation}\label{eqn:Lip1}
\big|x_k - \iso(y)_k\big|
\leq \min_{1\leq m\leq k\wedge (n-k+1)}\left\{\frac{\Lip(m-1)}{2n} +   \sqrt{\frac{2\sigma^2\log\left(\frac{n^2+n}{\delta}\right)}{m}}\right\},
\end{equation}
where the term $\frac{\Lip (m-1)}{2n}$ is a bound on $\left(\overline{x}_{k:(k+m-1)}-x_k\right)$ and
$\left(x_k -\overline{x}_{(k-m+1):k}\right)$ when $x$ is $\Lip$-Lipschitz.
It's easy to see that the optimal scaling is achieved by taking $m= \left\lceil \left(\frac{n \sqrt{\sigma^2 \log\left(\frac{n^2+n}{\delta}\right)}}{\Lip}\right)^{2/3}\right\rceil$, in which case we obtain the bound
\begin{equation}\label{eqn:Lip}
\big|x_k - \iso(y)_k\big|
\leq 2\sqrt[3]{\frac{\Lip {\sigma^2 \log\left(\frac{n^2+n}{\delta}\right)}}{n}}
\end{equation}
for all $m\leq k\leq n-m+1$. (For indices $k$ nearer to the endpoints, we are forced to choose a smaller $m$,
and the scaling will be worse.)

We can also compute convergence rates in a more general setting,
where the signal $x$ is locally Lipschitz---its behavior may vary across different regions of the signal.
As discussed in \secref{intro_refs}, many papers in the literature consider asymptotic local convergence rates---local
 in the sense of giving {\em pointwise} error bounds, which for a single signal $x=(x_1,\dots,x_n)$,
 may be larger for indices $i$ falling within a region of the signal that is strictly increasing, and smaller for indices $i$ falling into a locally flat region. We would hope to see some interpolation between 
the $n^{-1/3}$ rate expected for a strictly increasing stretch of the signal, as in~\eqnref{Lip},
versus the improved parametric rate of $n^{-1/2}$ in a locally constant region as in~\eqnref{half}.
 
 Our confidence bands can also be viewed as providing
 error bounds that are local in this sense, i.e.~that adapt to the local behavior of the signal $x$ as we move from index $i=1$ to $i=n$.
To make this more precise, we will  show how our bounds scale locally with the sample size $n$ to obtain the $n^{-1/3}$ and $n^{-1/2}$ rates
described above.
Consider any monotone signal $x$. 
Suppose the signal $x$ is locally constant near $k$, with
$x_{k-cn+1}=\dots = x_k = \dots = x_{k+cn-1}$ for some positive constant $c>0$.
Then our bound~\eqnref{subg_theory_eiso}
applied with $m=cn$ yields
\begin{equation}\label{eqn:predict_onehalf}\big|x_k - \iso(y)_k\big| \lesssim \sqrt{\frac{\sigma^2\log(n)}{n}}.\end{equation}
For other indices, however, where the signal is locally strictly increasing with a Lipschitz constant $\Lip$, then 
taking $m\sim\left(\frac{\sigma^2 n \log(n)}{\Lip}\right)^{2/3}$ yields the $n^{-1/3}$ scaling obtained
above in~\eqnref{Lip}.
It is of course also possible to achieve an interpolation between the $n^{-1/2}$ and
$n^{-1/3}$ rates via our results, as well.

Many works in the literature consider the local adaptivity problem in an asymptotic setting; here
we will describe the results of \citet{cator2011adaptivity}. 
Consider an asymptotic setting where the signal $x=(x_1,\dots,x_n)$ comes from measuring (at $n$ many points) a
monotone function $f:[0,1]\rightarrow \R$, and we are interested in the local convergence rate at some
fixed $t\in(0,1)$.  \citet{cator2011adaptivity} show that if the first $\alpha$ derivatives of $f$ at $t$ satisfy $f^{(1)}(t)=\dots = f^{(\alpha-1)}(t)=0$  and $f^{(\alpha)}(t)>0$, 
the convergence rate for estimating $f(t)$ scales as $n^{-\alpha/(2\alpha+1)}$. 
In particular, if $\alpha=1$ ($f$ is strictly increasing at $t$)
then they obtain the $n^{-1/3}$ rate seen before,
 while if $\alpha=\infty$ ($f$ is locally constant near $t$) then they obtain the faster
 parametric $n^{-1/2}$ rate. 
Of course, any $\alpha$ in between $1$ and $\infty$ will produce some power of $n$ between these two.
Our work can be viewed as a finite-sample version of these types of results.

\subsection{Convergence rates in the $\ell_2$ norm}
We next show that the tools developed in this paper can be used to yield a bound on the $\ell_2$ error, achieving the same $n^{-1/3}$ scaling
as in \citet{chatterjee2015risk}. While achieving an $n^{-1/3}$ elementwise requires a Lipschitz condition on the signal (as in our result~\eqnref{Lip} above),
here we do not assume any Lipschitz conditions and require only a bound on the total variation,
\[V \coloneqq \iso(x)_n - \iso(x)_1.\]
Our proof uses similar techniques as \citet{chatterjee2015risk}'s result.

\begin{theorem}\label{thm:L2}
For any signal $x\in\R^n$, under the subgaussian noise model~\eqnref{stat_model}, we have
\[\frac{1}{n}\norm{\iso(y) - \iso(x)}^2_2 \leq  48\left(\frac{V\sigma^2\log(2n)}{n}\right)^{2/3} +\frac{96\sigma^2\log^2(2n)}{n}.\]
\end{theorem}
As long as $n\gg \frac{\sigma^2\log^4(2n)}{V^2}$, the first term is the dominant one, matching the result of \citet[Theorem 4.1]{chatterjee2015risk} with a slight improvement in the log term. (The constants in this result are of course far from optimal.) This result is 
proved in \appref{proof_L2}.

\subsection{Simulation: local adaptivity}
To demonstrate this local adaptivity in practice, we run a simple simulation. The signal is generated from an underlying function $f(t)$
defined over $t\in[0,1]$, with
\[f(t) = \begin{cases}-10,&0\leq t\leq 0.3,\\
\text{linearly increasing from $-10$ to $10$},&0.3\leq t\leq 0.7,\\
10,&0.7\leq t \leq 1,\end{cases}\]
as illustrated in \figref{sim_signal}(a).
For a fixed sample size $n$, we set $x_i = f\left(\frac{i}{n+1}\right)$ and $y_i = x_i + N(0,1)$. We then compute a data-adaptive confidence band
as given in~\eqnref{subg_adapt_eiso}, with known noise level $\sigma=1$, with target coverage level $1-\delta = 0.9$, and with $\eiso=0$ as the signal $x$ is known to be monotone.
For sample size $n=1000$, the resulting estimate $\iso(y)$ and confidence band are illustrated in \figref{sim_signal}(b).

We then repeat this experiment at sample sizes $n=700,701,702,\dots,1000$. For each sample size $n$, we take the mean width of the confidence
band averaged over (a) the locally constant (``flat'') regions of the signal, defined by all indices $i$ corresponding to values $t\in[0.1,0.2]\cup[0.8,0.9]$, and (b) a strictly increasing
region, defined by indices $i$ corresponding to $t\in[0.4,0.6]$. (These regions are illustrated in \figref{sim_signal}(a).)

\begin{figure}
\centering
\begin{tabular}{cc}
\includegraphics[width=0.5\textwidth]{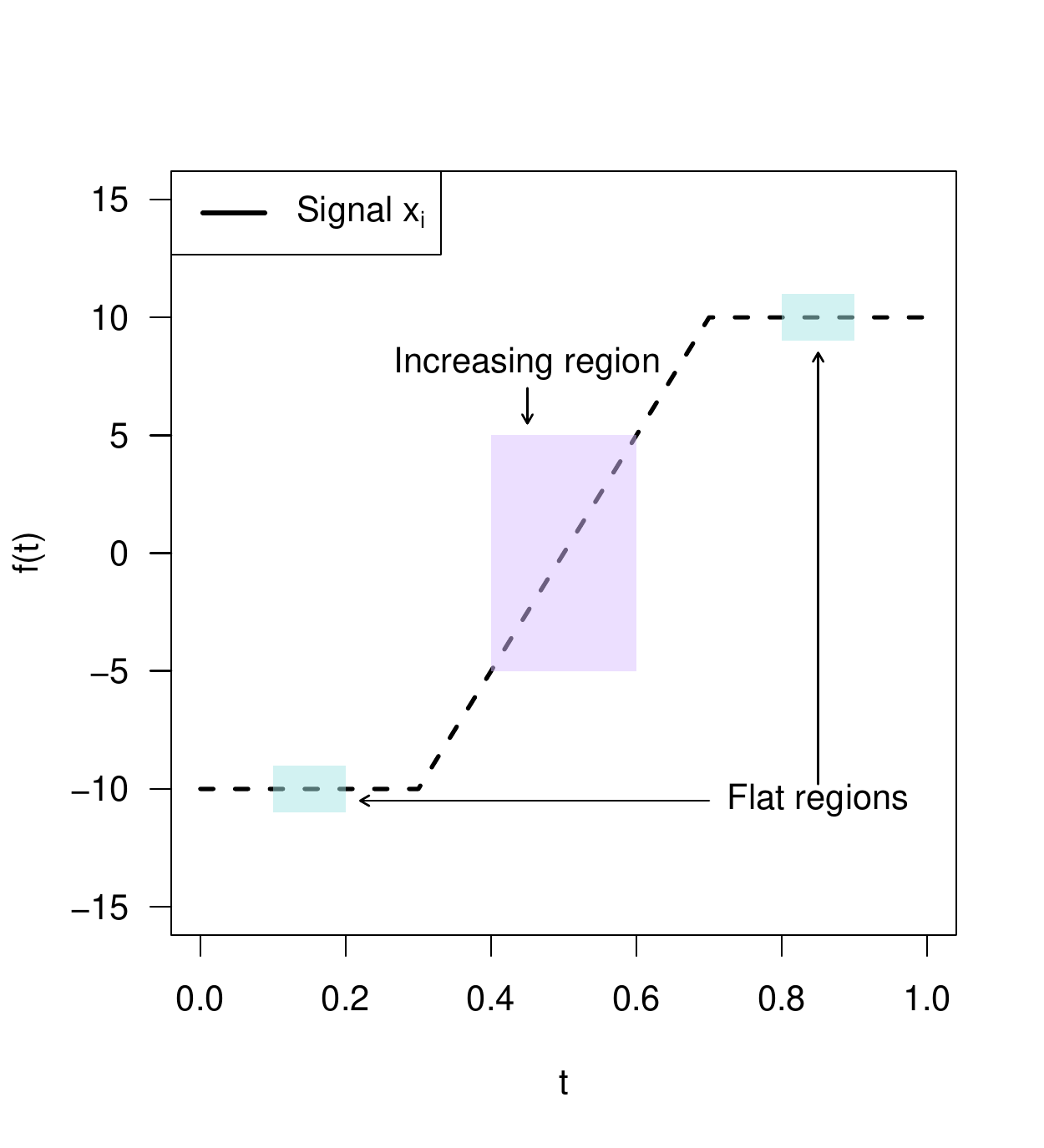}&
\includegraphics[width=0.5\textwidth]{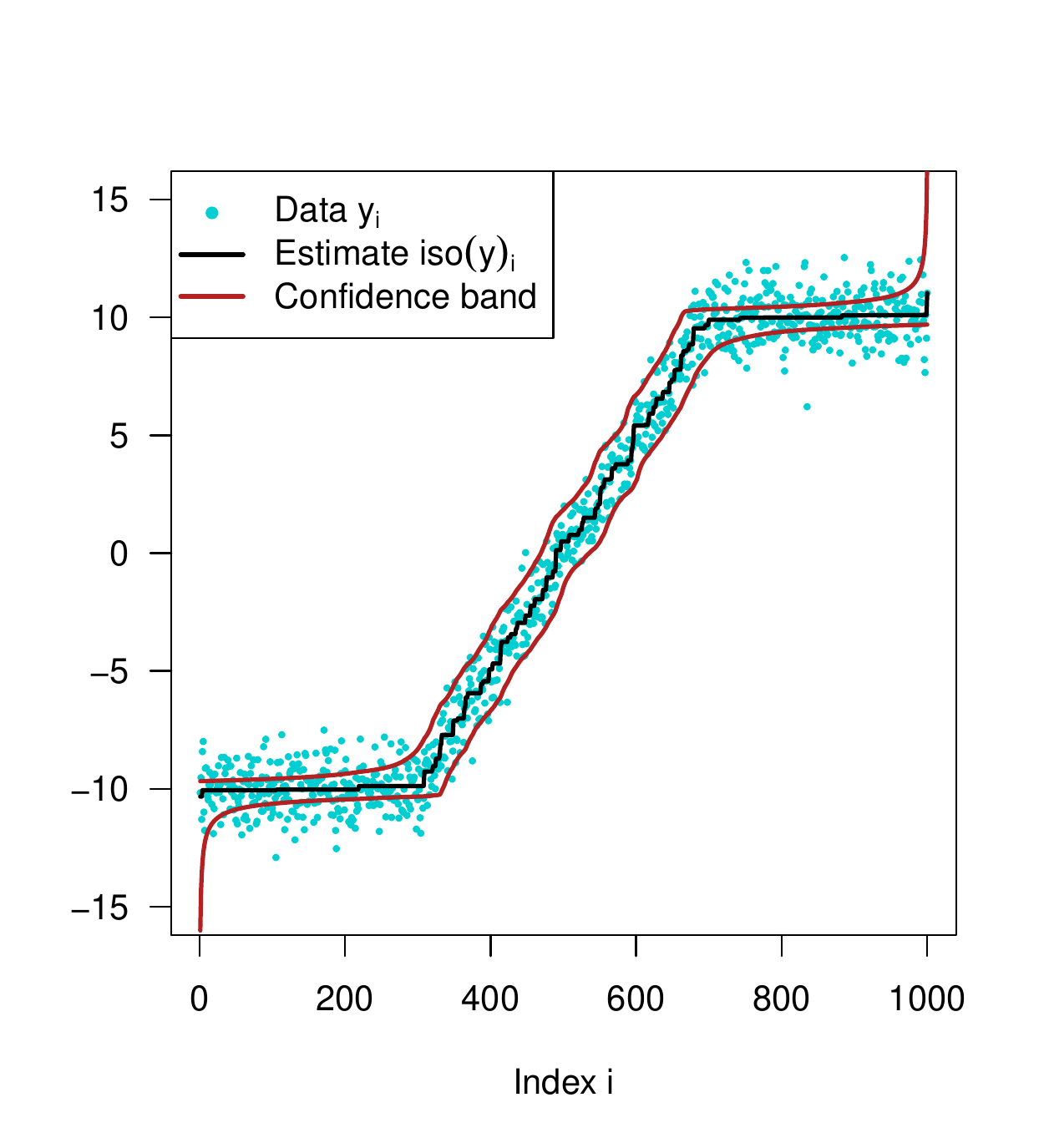}\\
(a) Signal&(b) Results for $n=1000$
\end{tabular}
\caption{(a) The function $f(t)$ used to generate signals $x\in\R^n$ for various $n$, with flat and increasing regions highlighted.
(b) At sample size $n=1000$, the observed data $y$, estimated signal $\iso(y)$, and data-adaptive confidence band computed as in~\eqnref{subg_adapt}.}
\label{fig:sim_signal}
\end{figure}

Our theory predicts that the mean confidence band width scales as $\sim \sqrt{\frac{\log(n)}{n}}$ in the flat regions, and $\sim \sqrt[3]{\frac{\log(n)}{n}}$ in the increasing
region. To test this, we take a linear regression of the log of the mean confidence band width against $\log\left(\frac{n}{\log(n)}\right)$, 
and find a slope $\approx -1/2$ in the flat regions and $\approx -1/3$ in the increasing region, confirming our theory. These results are illustrated in \figref{sim_scaling}.

Note that our data-adaptive estimation bands given by \thmref{subg_adapt} are calculated
without using prior knowledge of the signal's local behavior (locally constant / locally Lipschitz)---the confidence bands computed
in \thmref{subg_adapt} are able to adapt to this unknown structure automatically.

 We next check the empirical coverage level of these confidence bands. Ideally we would want to see that, over repeated simulations, the true monotone sequence $x=(x_1,\dots,x_n)$ lies entirely in the band roughly $1-\delta = 90\%$ of the time. While our theory guarantees that coverage will hold with probability {\em at least} 90\%, 
our bounds are of course somewhat conservative. We observe empirically that the coverage is in fact too high---it is essentially 100\%---but nonetheless, the width
of the confidence band is not too conservative. In particular, shrinking the width of the confidence band by a factor of $\approx 0.855$ empirically leads to 
achieving the target 90\% coverage level; in other
words, our confidence bands are around 17\% too wide. (Of course, this ratio is specific to our choice of data distribution, and is likely to vary in different settings.)

\begin{figure}
\centering
\begin{tabular}{cc}
\includegraphics[width=0.5\textwidth]{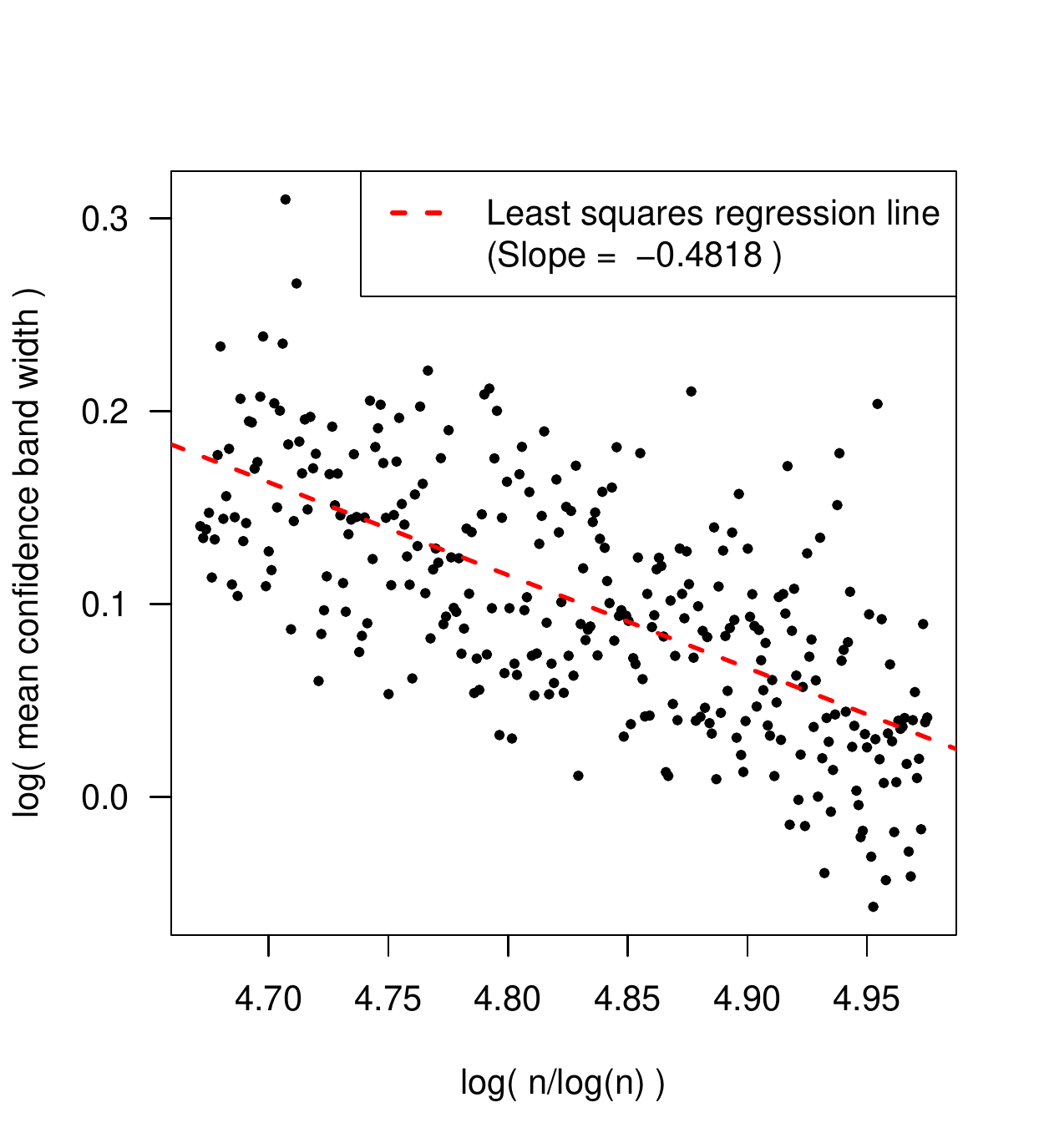}&
\includegraphics[width=0.5\textwidth]{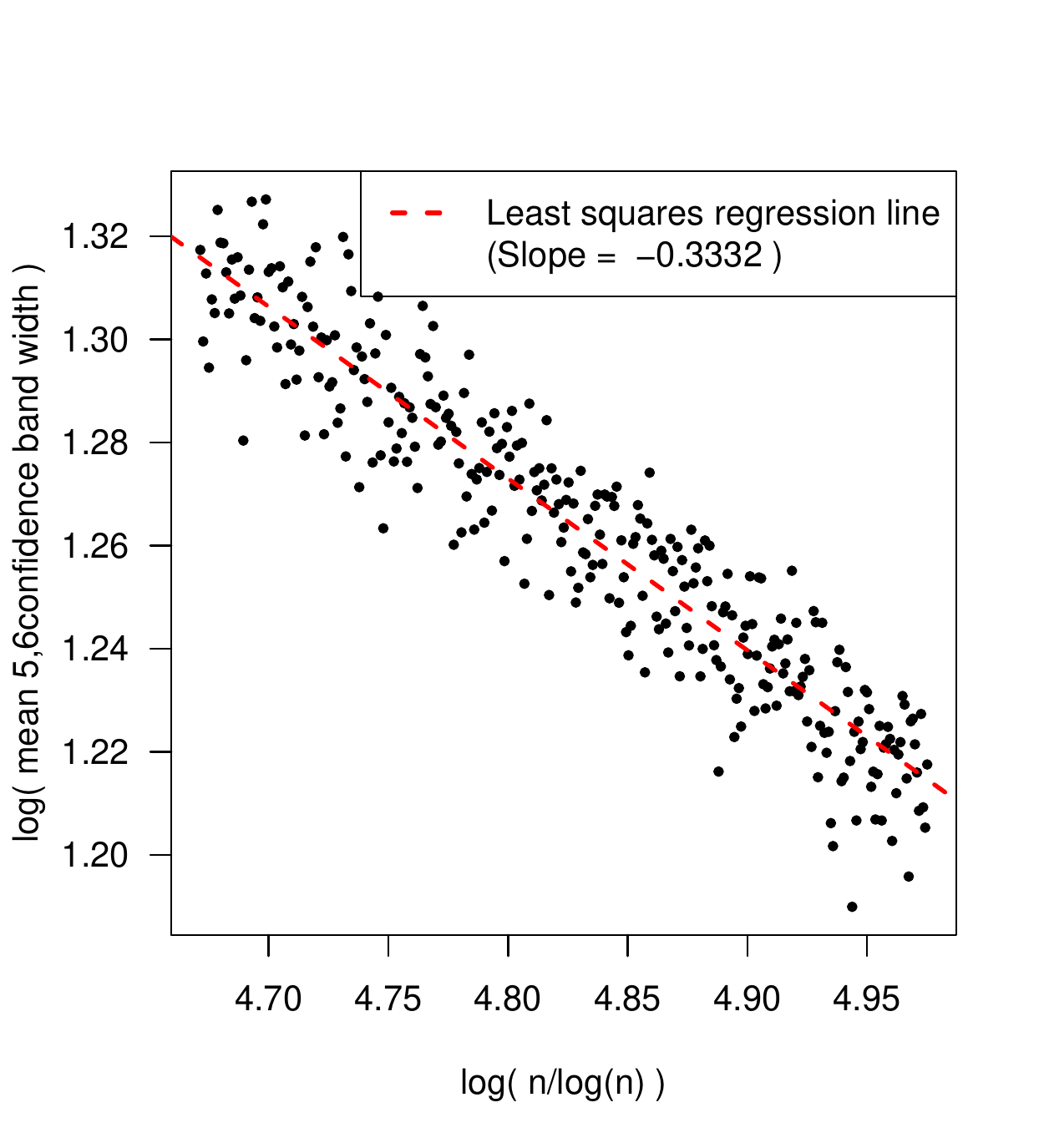}\\
(a) Results for flat regions &(b) Results for increasing region
\end{tabular}
\caption{ For each sample size $700\leq n\leq 1000$, log mean width of the confidence band  over a  region.
 (a) Flat region: $t\in[0.1,0.2]\cup[0.8,0.9]$, where  slope $\approx-1/2$, i.e.~pointwise error scales as $(n/\log(n))^{-1/2}$, as predicted in~\eqnref{predict_onehalf}. (b) Increasing region: $t\in[0.4,0.6]$, where slope  $\approx-1/3$, as predicted in~\eqnref{Lip}.}
\label{fig:sim_scaling}
\end{figure}

\section{Density estimation}
As a second application of the tools developed for the sliding window norm, we consider the problem of estimating a monotone nonincreasing
density $g$ on the interval $[0,1]$, using $n$ samples drawn from this density.

Let $Z_1,\dots,Z_n\iidsim g$ be $n$ samples drawn from the target density $f$, sorted into an ordered list $Z_{(1)}\leq \dots \leq Z_{(n)}$.
The Grenander estimator for the monotone density $g$ is defined as follows. Let $\Gh$ be the empirical cumulative distribution function for this sample,
\[\Gh(t) = \frac{1}{n}\sum_{i=1}^n \One{Z_i \leq t},\]
and let $\Ggren$ be the minimal concave upper bound on $\Gh$. Finally, define the Grenander estimator of the density, denoted by $\ggren$,
as the left-continuous piecewise constant first derivative of $\Ggren$. This process is illustrated in \figref{grenander}. 
It is known (\citet{robertsonorder}) that $\ggren$ can be computed with a simple isotonic projection of a sequence. Namely, 
for $i=1,\dots,n$, let $y_i = n(Z_{(i)} - Z_{(i-1)})$ where we set
$Z_{(0)}\coloneqq 0$, and calculate the isotonic projection $\iso(y)$. Then the Grenander estimator is given by
\begin{equation}\label{eqn:fg}
\ggren = \begin{cases}1/\iso(y)_1,&0\leq t\leq Z_{(1)},\\
1/\iso(y)_2,&Z_{(1)}<t\leq Z_{(2)},\\
\dots\\
1/\iso(y)_n,&Z_{(n-1)}<t\leq Z_{(n)},\\
0,&Z_{(n)}<t\leq 1.\end{cases}\end{equation}

\begin{figure}
\centering
\begin{tabular}{cc}
\includegraphics[width=0.5\textwidth]{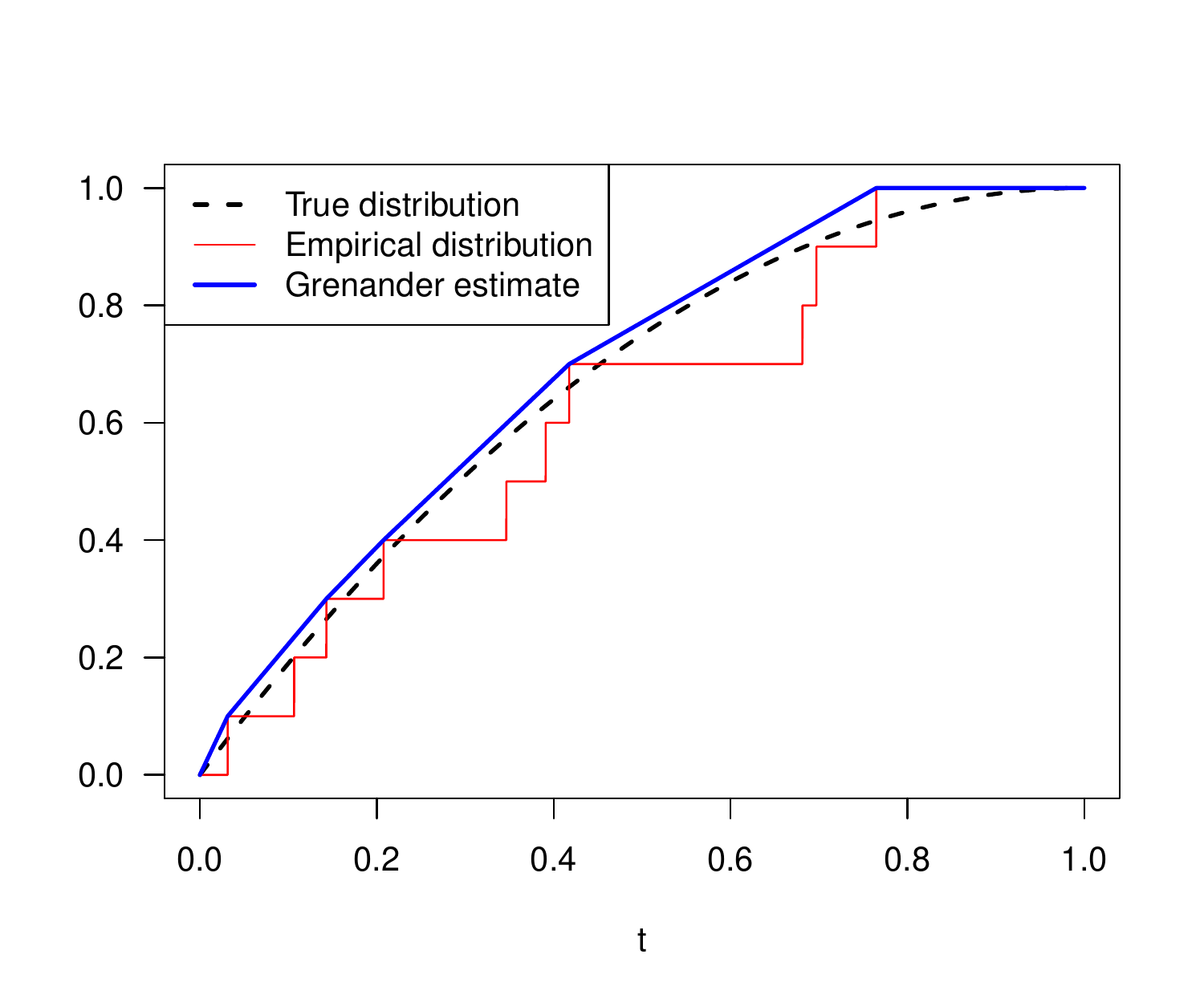}&
\includegraphics[width=0.5\textwidth]{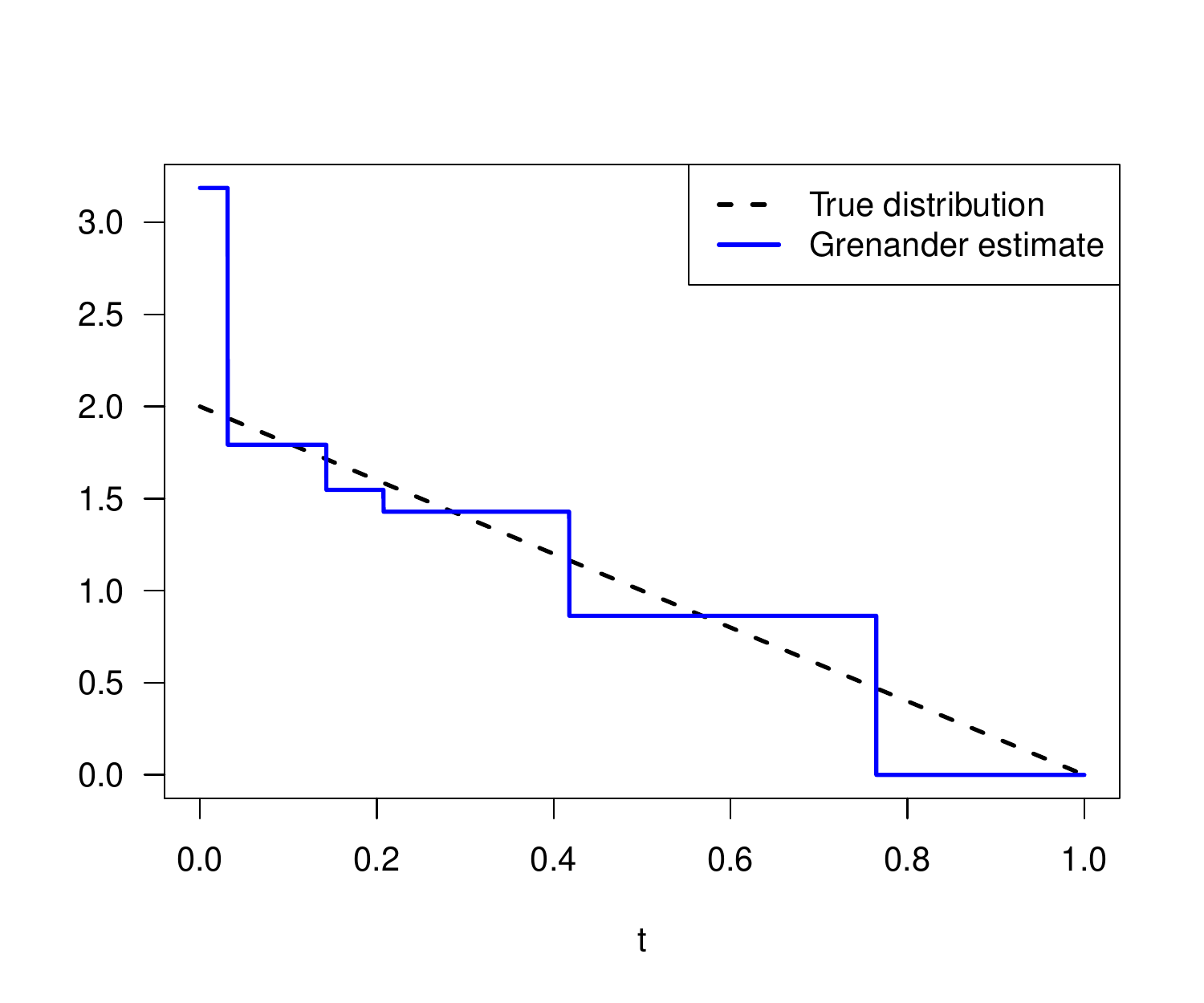}\\
(a) Cumulative distribution function &(b) Density
\end{tabular}
\caption{Illustration of the Grenander estimator for a monotone decreasing density.}
\label{fig:grenander}
\end{figure}

If we assume that $f$ is Lipschitz and lower-bounded, then
our error bounds for isotonic regression transfer easily into this setting, yielding the following theorem (proved in \appref{proofs_density}):
\begin{theorem}\label{thm:density}
Let $g:[0,1]\rightarrow [\cmin,\infty)$ be a nonincreasing $\Lip$-Lipschitz density, let $Z_1,\dots,Z_n$ be \iid~draws from $g$,
and define the Grenander estimator $\ggren$ as in~\eqnref{fg}.
Then for any $\delta>0$, if
\[\Delta \coloneqq  9\left(\frac{1}{c} + \frac{\Lip}{2c^3}\right) \sqrt[3]{\frac{\log((n^2+n)/\delta)}{n}} \leq \frac{1}{\cmin+\Lip},\]
then
\[\PP{\sup_{\Delta\leq t \leq 1-\Delta}\big|g(t) - \ggren(t)\big| \leq \frac{\Delta }{\frac{1}{\cmin+\Lip} \cdot \left(\frac{1}{\cmin+\Lip} - \Delta\right)}} \geq 1-\delta.\]
\end{theorem}

This result is similar to that of \citet{durot2012limit}, which also obtains a $(n/\log(n))^{-1/3}$ convergence rate uniformly over $t$ (although in their work, $t$ is allowed to be slightly closer to the endpoints, by a log factor). Their results are asymptotic, while our work gives a finite-sample
guarantee.
As mentioned in \secref{intro_refs}, \citet{cator2011adaptivity} also derives locally adaptive error bounds whose scaling depends on the local Lipschitz
behavior or local derivatives of $f$. Our locally adaptive results for sequences may also be applied here to  obtain a locally adaptive confidence band on the density $g$, but we do not give details
here.

\section{Proof for contractive isotonic projection}\label{sec:proof_nuna}
In this section, we prove our main result \thmref{iso_contract} showing that, for any semi-norm, the nonincreasing-under-neighbor-averaging (NUNA)
property is necessary and sufficient for isotonic projection to be contractive under this semi-norm.

Before proving the theorem, we introduce a few definitions.
First, for any index $i=1,\dots,n-1$, we define the matrix
\begin{equation}\label{eqn:Ai}
A_i = \left(\begin{array}{cccc}
\ident_{i-1} & 0 & 0 & 0 \\
0 & \nicefrac{1}{2} & \nicefrac{1}{2} & 0\\
0 & \nicefrac{1}{2} & \nicefrac{1}{2} & 0\\
0 & 0 & 0 & \ident_{n-i-1}\end{array}\right)\in\R^{n\times n} ,\end{equation}
which averages entries $i$ and $i+1$. That is,
\[A_i x = \left(x_1,\dots,x_{i-1},\frac{x_i + x_{i+1}}{2},\frac{x_i + x_{i+1}}{2},x_{i+2},\dots,x_n\right).\]
We also define an algorithm for isotonic projection that differs
from PAVA, and in fact does not converge in finite time, but is useful for the purpose
of theoretical analysis. For any $x\in\R^n$ and any index $i=1,\dots,n-1$, define
\[\iso_i(x) = \begin{cases}x,&\text{ if }x_i\leq x_{i+1},\\
A_i x,
&\text{ if }x_i>x_{i+1}.\end{cases}\]
 In other words, if neighboring entries $i$ and $i+1$ violate the monotonicity constraint,
then we average them.
The following lemma shows that, by iterating this step infinitely many times (while cycling
through the indices $i=1,\dots,n-1$), we converge
to the isotonic projection of $x$.
\begin{lemma}\label{lem:slow_pava}
Fix any $x=x^{(0)}\in\R^n$, and define
\begin{equation}\label{eqn:slow_pava}
x^{(t)} = \iso_{i_t}(x^{(t-1)}) \text{ where }i_t = 1 + \textnormal{mod}(t-1,n-1)\text{ for $t=1,2,3,\dots$}.\end{equation}
Then
\[\lim_{t\rightarrow\infty} x^{(t)} = \iso(x).\]
\end{lemma}
With this slow projection algorithm in place, we turn to the proof of our theorem.

\begin{proof}[Proof of \thmref{iso_contract}]
First suppose that $\norm{\cdot}$ satisfies the NUNA property.
We will prove that, for any $x,y\in\R^n$ and any index $i=1,\dots,n-1$,
\begin{equation}\label{eqn:iso_contract_step}\norm{\iso_i(x)-\iso_i(y)}\leq\norm{x-y}.\end{equation}
If this is true, then by \lemref{slow_pava}, this is sufficient
to see that isotonic projection is contractive with respect to $\norm{\cdot}$,
since the map $x\mapsto\iso(x)$ is just a composition of (infinitely many)
steps of the form $x\mapsto\iso_i(x)$. More concretely, defining $x^{(t)}$ and $y^{(t)}$ as in
 \lemref{slow_pava},~\eqnref{iso_contract_step} proves that $\norm{x^{(t)}-y^{(t)}}\leq \norm{x^{(t-1)} - y^{(t-1)}}$
 for each $t\geq1$. Applying this inductively proves that $\norm{x^{(t)}-y^{(t)}}\leq \norm{x-y}$ for all $t\geq 1$,
and then taking the limit as $t\rightarrow\infty$, we obtain $\norm{\iso(x)-\iso(y)}\leq \norm{x-y}$.

Now we turn to proving~\eqnref{iso_contract_step}. We will split into four cases.
\begin{itemize}
\item Case 1: $x_i \leq x_{i+1}$ and $y_i \leq y_{i+1}$. In this
case, $\iso_i(x)=x$ and $\iso_i(y) = y$, and so trivially,
\[\norm{\iso_i(x) - \iso_i(y)}=\norm{x-y}.\]
\item Case 2: $x_i > x_{i+1}$ and $y_i > y_{i+1}$. In this case,
we have
\[\big[\iso_i(x)\big]_i = \big[\iso_i(x)\big]_{i+1} = \frac{x_i + x_{i+1}}{2}\]
and
\[\big[\iso_i(y)\big]_i = \big[\iso_i(y)\big]_{i+1} = \frac{y_i + y_{i+1}}{2},\]
while all entries $j\not\in\{i,i+1\}$ are unchanged.
Therefore, we can write
\[\iso_i(x)- \iso_i(y) = A_i \cdot (x-y).\]
Since $\norm{\cdot}$ satisfies the NUNA property, therefore,
\[\norm{\iso_i(x) - \iso_i(y)} = \norm{A_i \cdot (x-y)}\leq \norm{x-y}.\]
\item Case 3: $x_i\leq x_{i+1}$ and $y_i > y_{i+1}$. Let
\[t = \frac{y_i - y_{i+1}}{x_{i+1} - x_i + y_i - y_{i+1}}.\]
Note that $t\in [0,1]$ by the definition of this case.
A trivial calculation shows that
\[\big[\iso_i(x) - \iso_i(y)\big]_i = x_i - \frac{y_i + y_{i+1}}{2} = (1-t/2)\cdot (x_i - y_i) + t/2\cdot (x_{i+1} - y_{i+1})\]
and
\[\big[\iso_i(x) - \iso_i(y)\big]_{i+1} = x_{i+1} - \frac{y_i + y_{i+1}}{2} = t/2\cdot (x_i - y_i) + (1-t/2)\cdot (x_{i+1} - y_{i+1})\]
This means that we have
\[\iso_i(x) - \iso_i(y) = (1-t)\cdot (x-y) + t\cdot A_i\cdot (x-y),\]
and so
\[\norm{\iso_i(x) - \iso_i(y)}\leq (1-t)\cdot\norm{x-y} + t\cdot \norm{A_i\cdot (x-y)} \leq \norm{x-y},\]
since $\norm{\cdot}$ satisfies NUNA.
\item Case 4: $x_i> x_{i+1}$ and $y_i\leq y_{i+1}$. By symmetry, this is equivalent
to Case 3.
\end{itemize}
This proves~\eqnref{iso_contract_step}, and therefore, is sufficient to show that
isotonic projection is a contraction with respect to $\norm{\cdot}$.

Now we prove the converse. Suppose that $\norm{\cdot}$ does not satisfy NUNA.
Then we can find some $x$ and some $i$ such that
\[\norm{A_i x}
> \norm{x}.\]
 Without loss of generality we can assume $x_i\leq x_{i+1}$
(otherwise simply replace $x$ with $-x$---since $\norm{\cdot}$ is a norm, we will have $\norm{- A_i x} = \norm{A_i x} > \norm{x} = \norm{-x}$). 

Let $B=\max_{1\leq j\leq n-1}|x_j - x_{j+1}|$, and let $\Delta=x_{i+1}-x_i \in [0, B]$.

Now define
\[y = (\Delta-(i-1)B,\Delta-(i-2)B,\dots,\Delta-B,\underbrace{\Delta}_{\text{entry $i$}},\underbrace{0}_{\text{entry $i+1$}},B,2B,\dots,(n-i-1)B),\]
and $z=y+x$.
We can check that $\iso(z)=z$, since
\[z_{j+1} - z_j = \begin{cases} x_{j+1} - x_j + B \geq 0,&\text{ if $j\neq i$},\\
x_{i+1} - x_i - \Delta = 0,&\text{ if $j=i$.}
\end{cases}\]
On the other hand, using the fact that $0\leq\Delta\leq B$, we have
\[\iso(y) = \left(\Delta-(i-1)B,\Delta-(i-2)B,\dots,\Delta-B,\frac{\Delta}{2},\frac{\Delta}{2},B,2B,\dots,(n-i-1)B\right),\]
and so
\begin{multline*}
\iso(z)-\iso(y) = (z-y) + (y-\iso(y)) = x + \left(0,\dots,0,\frac{\Delta}{2},-\frac{\Delta}{2},0,\dots,0\right)\\
 =\left(x_1,x_2,\dots,x_{i-1},\frac{x_i + x_{i+1}}{2}  ,\frac{x_i + x_{i+1}}{2},x_{i+2},x_{i+3},\dots,x_n\right) = A_i x.\end{multline*}
 Therefore, $\norm{\iso(z)-\iso(y)}>\norm{z-y}$, proving that
isotonic projection is not contractive with respect to $\norm{\cdot}$.

\end{proof}

\section{Discussion}
We study contraction properties of isotonic regression with an application on a novel sliding window norm. We then use these tools to construct data-adaptive estimation bands and obtain non-asymptotic uniform estimation bound of isotonic estimator. Our results are adaptive to the local behavior of the unknown signal, and can be used in a related density estimation problem.  The analysis tools we developed are potentially useful for other shape-restricted problems. 

We expect to apply our results on the high dimensional inference or calibration problems, where isotonic regression could serve as an important tool. 
Furthermore, ordering constraints, and more generally shape constraints, can take many different forms in various applications---for instance, two commonly studied constraints
 are isotonicity over a two-dimensional grid (in contrast to the one-dimensional ordering studied here), and convexity or concavity. 
It may be possible to extend the contraction results to a more general shape-constrained regression setting.
We leave these problems for future work.

\subsection*{Acknowledgements}
R.F.B.~was partially supported by an Alfred P.~Sloan Fellowship. The authors are grateful to Sabyasachi Chatterjee for 
helpful feedback on earlier drafts of this paper.
\bibliographystyle{plainnat}
\bibliography{draft}

\begin{thebibliography}{26}
\providecommand{\natexlab}[1]{#1}
\providecommand{\url}[1]{\texttt{#1}}
\expandafter\ifx\csname urlstyle\endcsname\relax
  \providecommand{\doi}[1]{doi: #1}\else
  \providecommand{\doi}{doi: \begingroup \urlstyle{rm}\Url}\fi

\bibitem[Balabdaoui et~al.(2011)Balabdaoui, Jankowski, Pavlides, Seregin, and
  Wellner]{balabdaoui2011grenander}
Fadoua Balabdaoui, Hanna Jankowski, Marios Pavlides, Arseni Seregin, and Jon
  Wellner.
\newblock On the grenander estimator at zero.
\newblock \emph{Statistica Sinica}, 21\penalty0 (2):\penalty0 873, 2011.

\bibitem[Barlow et~al.(1972)Barlow, Bartholomew, Bremner, and
  Brunk]{barlow1972statistical}
Richard~E Barlow, David~J Bartholomew, JM~Bremner, and H~Daniel Brunk.
\newblock \emph{Statistical inference under order restrictions: The theory and
  application of isotonic regression}.
\newblock Wiley New York, 1972.

\bibitem[Birg{\'e}(1987)]{birge1987estimating}
Lucien Birg{\'e}.
\newblock Estimating a density under order restrictions: Nonasymptotic minimax
  risk.
\newblock \emph{The Annals of Statistics}, pages 995--1012, 1987.

\bibitem[Birg{\'e} and Massart(1993)]{birge1993rates}
Lucien Birg{\'e} and Pascal Massart.
\newblock Rates of convergence for minimum contrast estimators.
\newblock \emph{Probability Theory and Related Fields}, 97\penalty0
  (1):\penalty0 113--150, 1993.

\bibitem[Bregman(1965)]{bregman1965method}
Lev~M Bregman.
\newblock The method of successive projection for finding a common point of
  convex sets(theorems for determining common point of convex sets by method of
  successive projection).
\newblock \emph{Soviet Mathematics}, 6:\penalty0 688--692, 1965.

\bibitem[Brunk(1969)]{brunk1969estimation}
Hugh~D Brunk.
\newblock \emph{Estimation of isotonic regression}.
\newblock University of Missouri-Columbia, 1969.

\bibitem[Carolan and Dykstra(1999)]{carolan1999asymptotic}
Chris Carolan and Richard Dykstra.
\newblock Asymptotic behavior of the grenander estimator at density flat
  regions.
\newblock \emph{Canadian Journal of Statistics}, 27\penalty0 (3):\penalty0
  557--566, 1999.

\bibitem[Cator(2011)]{cator2011adaptivity}
Eric Cator.
\newblock Adaptivity and optimality of the monotone least-squares estimator.
\newblock \emph{Bernoulli}, 17\penalty0 (2):\penalty0 714--735, 2011.

\bibitem[Chatterjee et~al.(2015)Chatterjee, Guntuboyina, Sen,
  et~al.]{chatterjee2015risk}
Sabyasachi Chatterjee, Adityanand Guntuboyina, Bodhisattva Sen, et~al.
\newblock On risk bounds in isotonic and other shape restricted regression
  problems.
\newblock \emph{The Annals of Statistics}, 43\penalty0 (4):\penalty0
  1774--1800, 2015.

\bibitem[Drton and Klivans(2010)]{drton2010geometric}
Mathias Drton and Caroline Klivans.
\newblock A geometric interpretation of the characteristic polynomial of
  reflection arrangements.
\newblock \emph{Proceedings of the American Mathematical Society}, 138\penalty0
  (8):\penalty0 2873--2887, 2010.

\bibitem[D{\"u}mbgen(2003)]{dumbgen2003optimal}
Lutz D{\"u}mbgen.
\newblock Optimal confidence bands for shape-restricted curves.
\newblock \emph{Bernoulli}, 9\penalty0 (3):\penalty0 423--449, 2003.

\bibitem[Durot et~al.(2012)Durot, Kulikov, and Lopuha{\"a}]{durot2012limit}
C{\'e}cile Durot, Vladimir~N Kulikov, and Hendrik~P Lopuha{\"a}.
\newblock The limit distribution of the $l_{\infty}$-error of grenander-type
  estimators.
\newblock \emph{The Annals of Statistics}, pages 1578--1608, 2012.

\bibitem[Gasser et~al.(1986)Gasser, Sroka, and
  Jennen-Steinmetz]{gasser1986residual}
Theo Gasser, Lothar Sroka, and Christine Jennen-Steinmetz.
\newblock Residual variance and residual pattern in nonlinear regression.
\newblock \emph{Biometrika}, pages 625--633, 1986.

\bibitem[Grenander(1956)]{grenander1956theory}
Ulf Grenander.
\newblock On the theory of mortality measurement: part ii.
\newblock \emph{Scandinavian Actuarial Journal}, 1956\penalty0 (2):\penalty0
  125--153, 1956.

\bibitem[Groeneboom(1984)]{groeneboom1984estimating}
Piet Groeneboom.
\newblock Estimating a monotone density.
\newblock \emph{Department of Mathematical Statistics}, \penalty0 (R
  8403):\penalty0 1--14, 1984.

\bibitem[Han(1988)]{han1988successive}
Shih-Ping Han.
\newblock A successive projection method.
\newblock \emph{Mathematical Programming}, 40\penalty0 (1-3):\penalty0 1--14,
  1988.

\bibitem[Jankowski(2014)]{jankowski2014convergence}
Hanna Jankowski.
\newblock Convergence of linear functionals of the grenander estimator under
  misspecification.
\newblock \emph{The Annals of Statistics}, 42\penalty0 (2):\penalty0 625--653,
  2014.

\bibitem[Meyer and Woodroofe(2000)]{meyer2000degrees}
Mary Meyer and Michael Woodroofe.
\newblock On the degrees of freedom in shape-restricted regression.
\newblock \emph{Annals of Statistics}, pages 1083--1104, 2000.

\bibitem[Rao(1969)]{rao1969estimation}
BLS~Prakasa Rao.
\newblock Estimation of a unimodal density.
\newblock \emph{Sankhy{\=a}: The Indian Journal of Statistics, Series A}, pages
  23--36, 1969.

\bibitem[Rice(1984)]{rice1984bandwidth}
John Rice.
\newblock Bandwidth choice for nonparametric regression.
\newblock \emph{The Annals of Statistics}, pages 1215--1230, 1984.

\bibitem[Robertson et~al.(1988)Robertson, Wright, and Dykstra]{robertsonorder}
Tim Robertson, Farrol~T Wright, and Richard~L Dykstra.
\newblock Order restricted statistical inference, 1988.

\bibitem[Van~de Geer(1990)]{van1990estimating}
Sara Van~de Geer.
\newblock Estimating a regression function.
\newblock \emph{The Annals of Statistics}, pages 907--924, 1990.

\bibitem[Van~de Geer(1993)]{van1993hellinger}
Sara Van~de Geer.
\newblock Hellinger-consistency of certain nonparametric maximum likelihood
  estimators.
\newblock \emph{The Annals of Statistics}, pages 14--44, 1993.

\bibitem[Wang and Chen(1996)]{wang1996l2risk}
Yazhen Wang and KS~Chen.
\newblock The l2risk of an isotonic estimate.
\newblock \emph{Communications in Statistics-Theory and Methods}, 25\penalty0
  (2):\penalty0 281--294, 1996.

\bibitem[Wright(1981)]{wright1981asymptotic}
Farrol~T Wright.
\newblock The asymptotic behavior of monotone regression estimates.
\newblock \emph{The Annals of Statistics}, 9\penalty0 (2):\penalty0 443--448,
  1981.

\bibitem[Zhang(2002)]{zhang2002risk}
Cun-Hui Zhang.
\newblock Risk bounds in isotonic regression.
\newblock \emph{The Annals of Statistics}, 30\penalty0 (2):\penalty0 528--555,
  2002.

\end{thebibliography}

\appendix
\section{Additional proofs}\label{app:proofs}

\subsection{Proof of $\ell_2$ error rate (\thmref{L2})}\label{app:proof_L2}

First we define the cube $A = [\iso(x)_1, \iso(x)_n]^n \subset \R^n$, and let $z = \pr{A}(\iso(y))$
be the projection of $\iso(y)$ to this cube, which is computed by truncating each entry $\iso(y)_i$ to the range $[\iso(x)_1,\iso(x)_n]$.
Note that $\iso(x)+z$ is now a monotone vector with range given by $(\iso(x)+z)_n - (\iso(x)+z)_1\leq 2V$.
Now, fix any integer $M\geq 1$ and find integers
\[0 = k_0< \dots <k_M=n\]
such that
\[\Big| (\iso(x) + z)_{k_m} - (\iso(x)+z)_{k_{m-1}+1}\Big| \leq \frac{2V}{M}\text{ for all $m=1,\dots,M$},\]
which we can find since the total variation of the vector $\iso(x)+z$ is bounded by $2V$ (Lemma 11.1 in \citet{chatterjee2015risk}).

For each $m=1,\dots,M$, let $I_m = \{k_{m-1}+1,\dots,k_m\}$ be the set of indices in the $m$th bin.
We can calculate
\begin{multline*}
\max_{i\in I_m}(z-\iso(x))_i -\min_{i\in I_m}(z-\iso(x))_i\\
\leq \max_{i\in I_m}z_i + \max_{i\in I_m}\iso(x)_i - \min_{i\in I_m} z_i - \min_{i\in I_m}\iso(x)_i\\
= (\iso(x)+z)_{k_m} - (\iso(x)+z)_{k_{m-1}+1} \leq \frac{2V}{M}.\end{multline*}
This implies that
\[ \norm{(z-\iso(x))_{I_m} -\overline{(z-\iso(x))}_{I_m}\cdot\ones_{I_m}}_2 \leq  \frac{2V}{M}\cdot\sqrt{k_m-k_{m-1}}.\]
We also have
\[\big|\overline{z}_{I_m} - \overline{\iso(x)}_{I_m}\big| \leq \frac{\sw{z-\iso(x)}}{\sqrt{k_m-k_{m-1}}}\]
by our choice of the sliding window norm. 
Next, by the triangle inequality we have
\begin{multline*}
\norm{z_{I_m} - \iso(x)_{I_m}}^2_2 \\
\leq \left(\norm{\overline{z}_{I_m}\cdot\ones_{I_m} - \overline{\iso(x)}_{I_m}\cdot\ones_{I_m}}_2  + \norm{(z-\iso(x))_{I_m} -\overline{(z-\iso(x))}_{I_m}\cdot\ones_{I_m}}_2 \right)^2\\
=\left(\big|\overline{z}_{I_m} - \overline{\iso(x)}_{I_m}\big|\cdot\sqrt{k_m-k_{m-1}} +  \norm{(z-\iso(x))_{I_m} -\overline{(z-\iso(x))}_{I_m}\cdot\ones_{I_m}}_2\right)^2\\
\leq \left( \sw{z-\iso(x)} + \frac{2V}{M}\cdot\sqrt{k_m-k_{m-1}}\right)^2 \leq 2\big(\sw{z-\iso(x)}\big)^2 + \frac{8V^2}{M^2}(k_m - k_{m-1}).
\end{multline*}
Therefore,
\[\norm{z-\iso(x)}^2_2 = \sum_{m=1}^M \norm{z_{I_m}-\iso(x)_{I_m}}^2_2 \leq 2M\big(\sw{z-\iso(x)}\big)^2 + \frac{8V^2}{M^2}\sum_{m=1}^M (k_m- k_{m-1}).\]
Since $\sum_{m=1}^M (k_m-k_{m-1}) = n$ trivially, we can simplify this to
\[\norm{z-\iso(x)}^2_2 \leq2M\big(\sw{z-\iso(x)}\big)^2 + \frac{8V^2n}{M^2}.\]
Now, since $z$ is the projection of $\iso(y)$ to the range $[\iso(x)_1,\iso(x)_n]$, it follows trivially that $\sw{z-\iso(x)}\leq \sw{\iso(x)-\iso(y)}$ and therefore is bounded by $\sw{x-y}$ by our contraction result. Therefore,
\[\norm{z-\iso(x)}^2_2 \leq2M\big(\sw{x-y}\big)^2 + \frac{8V^2n}{M^2}.\]
Next, we have
\begin{multline*}\norm{\iso(y) - \iso(x)}^2_2 \leq 2\norm{z-\iso(x)}^2_2 + 2\norm{\iso(y) - z}^2_2\\ = 2\norm{z-\iso(x)}^2_2 + 2\sum_{i=1}^n (z_i-\iso(y)_i)^2_+ + 2\sum_{i=1}^n (\iso(y)_i - z_i)^2_+\\
\leq 4M\big(\sw{x-y}\big)^2 + \frac{16V^2n}{M^2} + 2\sum_{i=1}^n (\iso(x)_1 -\iso(y)_i)^2_+ + 2\sum_{i=1}^n (\iso(y)_i - \iso(x)_n )^2_+.
\end{multline*}
It remains to bound these last two terms. First we bound $\sum_{i=1}^n (\iso(x)_1 -\iso(y)_i)^2_+$. If $\iso(y)_1 \geq \iso(x)_1$ then this term is zero, so now we focus on the case that $\iso(y)_1 < \iso(x)_1$. For any $1\leq j\leq i$, we have
\[\iso(x)_j - \iso(y)_j  \geq \iso(x)_1 - \iso(y)_j \]
and so
\[ \iso(x)_1 - \iso(y)_i \leq \overline{\iso(x)}_{1:i} - \overline{\iso(y)}_{1:i} \leq \frac{\sw{x-y}}{\sqrt{i}}.\]
Therefore,
\[\sum_{i=1}^n ( \iso(x)_1 - \iso(y)_i)^2_+ \leq \sum_{i=1}^n \left(\frac{\sw{x-y}}{\sqrt{i}}\right)^2 = \big(\sw{x-y}\big)^2 \cdot 2\log(2n),\]
by bounding the harmonic series $1 + \frac{1}{2}+\dots + \frac{1}{n}$. We also have $\sum_{i=1}^n (\iso(y)_i - \iso(x)_n )^2_+\leq 
 \big(\sw{x-y}\big)^2 \cdot 2\log(2n)$ by an identical argument. Combining everything, then,
\[\norm{\iso(y) - \iso(x)}^2_2\leq 4M\big(\sw{x-y}\big)^2 + \frac{16V^2n}{M^2} + 8\big(\sw{x-y}\big)^2\log(2n).\]
Setting
\[M =\lceil M_0 \rceil\text{ where }M_0 = \frac{2V^{2/3}n^{1/3}}{\big(\sw{x-y}\big)^{2/3}},\]
we obtain
\begin{multline*}\norm{\iso(y) - \iso(x)}^2_2\leq 4\left(\frac{2V^{2/3}n^{1/3}}{\big(\sw{x-y}\big)^{2/3}}+1\right)\big(\sw{x-y}\big)^2 \\+ \frac{16V^2n}{\left(\frac{2V^{2/3}n^{1/3}}{\big(\sw{x-y}\big)^{2/3}}\right)^2} + 8\big(\sw{x-y}\big)^2\log(2n).\end{multline*}
After simplifying (and assuming $n\geq 2$ to avoid triviality), this bound becomes
\[\frac{1}{n}\norm{\iso(y) - \iso(x)}^2_2\leq \frac{12V^{2/3}\big(\sw{x-y}\big)^{4/3}}{n^{2/3}}   +\frac{12\log(2n)}{n}\cdot  \big(\sw{x-y}\big)^2.\]
Applying \lemref{SW_subg}, we have
$\EE{  \big(\sw{x-y}\big)^2} \leq 8\sigma^2\log(2n)$
which implies that $\EE{  \big(\sw{x-y}\big)^{4/3}} \leq (8\sigma^2\log(2n) )^{2/3} = 4\sigma^{4/3}\log^{2/3}(2n)$.
Plugging this in, we then have
\[\EE{\frac{1}{n}\norm{\iso(y) - \iso(x)}^2_2}\leq 48\left(\frac{V\sigma^2\log(2n)}{n}\right)^{2/3} +\frac{96\sigma^2\log^2(2n)}{n}.\]

\subsection{Proof of density estimation result (\thmref{density})}\label{app:proofs_density}

Let $G(t) = \int_{s=0}^t g(s)\;\mathsf{d}s$ be the cumulative distribution function for the density $g$.
Since $g(t)\geq \cmin$ everywhere, this means that $G(t)$ is strictly increasing, and is therefore
invertible. Using this lower bound on $g$, and the assumption that $g$ is $\Lip$-Lipschitz, we can furthermore calculate
\begin{equation}\label{eqn:Gderivs}
0\leq (G^{-1})'(t) = \frac{1}{g(G^{-1}(t))} \leq \frac{1}{c} \ \text{ and } \ \big|(G^{-1})''(t)\big| = \left|\frac{-g'(G^{-1}(t))}{(g(G^{-1}(t)))^3}\right| \leq \frac{L}{c^3}.\end{equation}

Let
\[x_i =n\left( G^{-1}(i/n) - G^{-1}\big((i-1)/n\big)\right)\text{ and }y_i = n\big(Z_{(i)} - Z_{(i-1)}\big)\]
for $i=1,\dots,n$, where $Z_{(0)}\coloneqq 0$.
Note that $x$ gives the difference in quantiles of the distribution, while $y$ estimates
these gaps empirically.

The following lemma (proved in \appref{proofs_lemmas}) gives a concentration result on the $Z_{(i)}$'s:
\begin{lemma}\label{lem:spacing}
Let $Z_{(1)}\leq \dots\leq Z_{(n)}$ be the order statistics of $Z_1,\dots,Z_n\iidsim g$, where the density $g:[0,1]\rightarrow[\cmin,\infty)$ is $\Lip$-Lipschitz.
Then for any $\delta>0$, with probability at least $1-\delta$,
\begin{equation}\label{eqn:spacing1}\left|Z_{(i)} - G^{-1}(i/n)\right|  \leq     \frac{4}{\cmin}\cdot \sqrt{\frac{\log((n^2+n)/\delta)}{ n}}\end{equation}
$\text{ for all $1\leq i\leq n$}$, and
\begin{multline}\label{eqn:spacing2}\left|\big(Z_{(i)} - Z_{(j)}\big) -\left( G^{-1}(i/n) -  G^{-1}(j/n)\right)\right| \\
\leq \frac{\sqrt{3|i-j|\log((n^2+n)/\delta)}  +2\log((n^2+n)/\delta)}{\cmin n}+ \frac{4\Lip|i-j| \sqrt{\log((n^2+n)/\delta)}}{\cmin^3 n^{3/2}} \end{multline}
$\text{ for all $1\leq i<j\leq n$.}$
\end{lemma}
From now on, we assume that these bounds~\eqnref{spacing1} and~\eqnref{spacing2} both hold.
Plugging our definitions of $x$ and $y$ into these two bounds, this proves that
\[\big| \overline{x}_{i:j}-\overline{y}_{i:j}\big|   \leq \frac{\sqrt{3(j-i+1)\log((n^2+n)/\delta)}  +2\log((n^2+n)/\delta)}{\cmin \cdot  (j-i+1)}+ \frac{4\Lip}{\cmin^3}\cdot\sqrt{\frac{\log((n^2+n)/\delta)}{n}}\]
for all $1\leq i\leq j\leq n$. (If $i=1$ then we use the bound~\eqnref{spacing1} while if $i>1$ then we use the bound~\eqnref{spacing2}.)

Now, defining
\[\psi(i) = \frac{i}{ \frac{1}{\cmin}\cdot\left(\sqrt{3i\log((n^2+n)/\delta)}  +2\log((n^2+n)/\delta)\right) + \frac{4\Lip}{\cmin^3}\cdot i\cdot \sqrt{\frac{\log((n^2+n)/\delta)}{n}}},\]
we see that
\[\sw{x-y} =\max_{1\leq i\leq j\leq n}\big|\overline{x}_{i:j}-\overline{y}_{i:j}\big|\cdot\psi(j-i+1) \leq 1.\]
(Note that $\psi$ is nondecreasing and $i\mapsto i/\psi(i)$ is concave, as required by~\eqnref{psi}.)

Next we check that $x$ is a Lipschitz sequence. We have
\begin{multline*}n^{-1}\big(x_{i+1} - x_i \big)=  \left(G^{-1}\left(\frac{i+1}{n}\right) - G^{-1}(i/n) \right)+\left( G^{-1}\big((i-1)/n\big) - G^{-1}(i/n)\right) \\
= (G^{-1})'(i/n)\cdot \frac{1}{n} + \frac{1}{2}(G^{-1})''\left(\frac{i+s}{n}\right)\cdot\frac{1}{n^2}\\ + (G^{-1})'(i/n)\cdot -\frac{1}{n} + \frac{1}{2}(G^{-1})''\left(\frac{i-1+t}{n}\right)\cdot \frac{1}{n^2}\end{multline*}
for some $s,t\in[0,1]$, by Taylor's theorem. The first-order terms cancel, and we know by  \eqnref{Gderivs} that $(G^{-1})''$ is bounded by $\frac{\Lip}{\cmin^3}$.
Therefore, $x$ is $\frac{\Lip}{\cmin^3}$-Lipschitz. Finally, $x$ is monotone nondecreasing since $g$ is 
a monotone nonincreasing density.

We then apply the calculations~\eqnref{Lip1} for the Lipschitz signal setting
(with $\frac{\sw{x-y}}{\psi(m)}$ taking the place of $\sqrt{\frac{2\sigma^2\log\left(\frac{n^2+n}{\delta}\right)}{m}}$,
which was specific to the subgaussian noise model setting). We see that for any index $m\geq 1$,
\begin{multline*}
\max_{m\leq k\leq n-m+1} \big|x_k - \iso(y)_k\big| \leq \frac{\sw{x-y}}{\psi(m)} + \frac{\Lip(m-1)}{2n\cmin^3} \\
\leq  \frac{\frac{1}{\cmin}\cdot\left(\sqrt{3m\log((n^2+n)/\delta)}  +2\log((n^2+n)/\delta)\right) + \frac{4\Lip}{\cmin^3}\cdot m\cdot \sqrt{\frac{\log((n^2+n)/\delta)}{n}}}{m} + \frac{\Lip(m-1)}{2n\cmin^3} .
\end{multline*}
Set $m = \left\lceil \left(n \sqrt{\log((n^2+n)/\delta)}\right)^{2/3}\right\rceil$. Since $\Delta < \frac{1}{c+L} \leq 1$ we know $\log((n^2+n)/\delta)\leq n$,
so we can simplify the above bound to
 \[\max_{m\leq k\leq n-m+1} \big|x_k - \iso(y)_k\big| \leq  \left(\frac{4}{\cmin} + \frac{4.5\Lip}{\cmin^3}\right) \sqrt[3]{\frac{\log((n^2+n)/\delta)}{n}}.\]

Now we show how this uniform bound on the difference $x-y$,
translates to an error bound on the Grenander density estimator $\ggren$.
First, we check that $Z_{(m)}\leq \Delta$ and $Z_{(n-m+1)}\geq 1-\Delta$.
We have
\[1\geq \sw{x-y} \geq \psi(m)\cdot \big|\overline{x}_{1:m} - \overline{y}_{1:m}\big|
= \frac{\psi(m)}{m} \cdot n \cdot \big|G^{-1}(m/n) - Z_{(m)}\big|.\]
We also know that $G^{-1}(m/n)\leq \frac{m}{\cmin n}$ since $G^{-1}$ is $(1/\cmin)$-Lipschitz as calculated in~\eqnref{Gderivs},
and so
\begin{multline*}
Z_{(m)} \leq G^{-1}(m/n) + \frac{m}{n \psi(m)} \leq \\\frac{m}{\cmin n} + \frac{\sqrt{3m\log((n^2+n)/\delta)} + 2\log((n^2+n)/\delta)}{\cmin n} + \frac{4\Lip  m \sqrt{\log((n^2+n)/\delta)} }{\cmin^3 n^{3/2}}\\
\leq \left(\frac{5}{\cmin} + \frac{4\Lip}{\cmin^3}\right) \sqrt[3]{\frac{\log((n^2+n)/\delta)}{n}}\leq \Delta,
\end{multline*}
using the fact that $\log((n^2+n)/\delta)\leq n$ as before.
Similarly  $Z_{(n-m+1)}\geq 1-\Delta$.

Next, for any $t$ with $\Delta\leq t\leq 1-\Delta$, find index $k$ such that
\[Z_{(k-1)}< t \leq Z_{(k)}.\]
By the work above we will have $m\leq k \leq n-m+1$. Then 
$\ggren(t) = \frac{1}{\iso(y)_k}$
by definition of the Grenander estimator. Therefore, we have
\[
\big|\ggren(t) - g(t)\big|
= \left|\frac{1}{\iso(y)_k} - g(t)\right|\\
\leq  \left|\frac{1}{\iso(y)_k} - \frac{1}{x_k}\right| + \left|\frac{1}{x_k} -  g(t)\right|.\]
By \lemref{spacing}, we know 
\begin{equation*}
Z_{(k)} \leq G^{-1}\left(\frac{k}{n}\right) + \frac{4}{\cmin} \sqrt[3]{\frac{\log((n^2+n)/\delta)}{n}} 
\end{equation*}
and
\[
Z_{(k-1)} \geq G^{-1}\left(\frac{k-1}{n}\right) - \frac{4}{\cmin} \sqrt[3]{\frac{\log((n^2+n)/\delta)}{n}}
\]
so we have
\[  G^{-1}\left(\frac{k-1}{n}\right) - \frac{4}{\cmin} \sqrt[3]{\frac{\log((n^2+n)/\delta)}{n}} < t \leq G^{-1}\left(\frac{k}{n}\right) + \frac{4}{\cmin} \sqrt[3]{\frac{\log((n^2+n)/\delta)}{n}} \]
We calculate
\[x_k = n\left(G^{-1}\left(\frac{k}{n}\right)  - G^{-1}\left(\frac{k-1}{n}\right) \right) = n (G^{-1})'\left(\frac{k-1+s}{n}\right)\cdot \frac{1}{n} = \frac{1}{g\left(G^{-1}\left(\frac{k-1+s}{n}\right)\right)}\]
by Taylor's theorem for some $s\in[0,1]$, and so
\begin{multline*}\left|\frac{1}{x_k} -  g(t)\right| = \left|g\left(G^{-1}\left(\frac{k-1+s}{n}\right)\right) - g(t)\right|
\leq \Lip\cdot\left|G^{-1}\left(\frac{k-1+s}{n}\right) - t\right|\\ 
\leq \Lip \left|G^{-1}\left(\frac{k}{n}\right)  - G^{-1}\left(\frac{k-1}{n}\right)\right| +  \frac{4L}{\cmin} \sqrt[3]{\frac{\log((n^2+n)/\delta)}{n}} \\ \leq \frac{\Lip}{\cmin n}  +  \frac{4L}{\cmin} \sqrt[3]{\frac{\log((n^2+n)/\delta)}{n}},
\end{multline*}
since $g$ is $\Lip$-Lipschitz and $G^{-1}$ is $(1/\cmin)$-Lipschitz, as proved before.
Finally, 
\begin{equation*}
\begin{split}
& \left|\frac{1}{\iso(y)_k} - \frac{1}{x_k}\right|
 = \frac{\left|x_k - \iso(y)_k\right|}{x_k \cdot \iso(y)_k} \\ &
\leq \frac{ \left(\frac{4}{\cmin} + \frac{4.5\Lip}{\cmin^3}\right) \sqrt[3]{\frac{\log((n^2+n)/\delta)}{n}}}{x_k \cdot \iso(y)_k}
\leq \frac{ \left(\frac{4}{\cmin} + \frac{4.5\Lip}{\cmin^3}\right) \sqrt[3]{\frac{\log((n^2+n)/\delta)}{n}}}{x_k \cdot \left(x_k - \left(\frac{4}{\cmin} + \frac{4.5\Lip}{\cmin^3}\right) \sqrt[3]{\frac{\log((n^2+n)/\delta)}{n}}\right)},
\end{split}
\end{equation*}
from the bound on $\big|x_k - \iso(y)_k\big|$ above.
And, we know that
$x_k =  \frac{1}{g\left(G^{-1}\left(\frac{k-1+s}{n}\right)\right)}$
for some $s\in[0,1]$ as above, so $x_k \geq \frac{1}{\max_{s\in[0,1]}g(s)}$. Now, since $g$ is
lower-bounded by $\cmin$ and is $\Lip$-Lipschitz, we see that $g(s)\leq \cmin+\Lip$, and so $x_k\geq \frac{1}{\cmin+\Lip}$.
 Combining everything, 
\begin{multline*}
\big|\ggren(t) - g(t)\big|\leq \\ \frac{ \left(\frac{4}{\cmin} + \frac{4.5\Lip}{\cmin^3}\right) \sqrt[3]{\frac{\log((n^2+n)/\delta)}{n}}}{\frac{1}{\cmin+\Lip} \cdot \left(\frac{1}{\cmin+\Lip} - \left(\frac{4}{\cmin} + \frac{4.5\Lip}{\cmin^3}\right) \sqrt[3]{\frac{\log((n^2+n)/\delta)}{n}}\right)} + \frac{\Lip}{\cmin n} + \frac{4\Lip}{\cmin}\sqrt[3]{\frac{\log((n^2+n)/\delta)}{n}} \\ 
\leq \frac{ \left(\frac{1}{c}\left(4+\frac{5L}{c+L}\right) + \frac{4.5\Lip}{c^3}\right)\sqrt[3]{\frac{\log((n^2+n)/\delta)}{n}}}{\frac{1}{\cmin+\Lip} \cdot \left(\frac{1}{\cmin+\Lip} - \left(\frac{4}{\cmin} + \frac{4.5\Lip}{\cmin^3}\right) \sqrt[3]{\frac{\log((n^2+n)/\delta)}{n}}\right)}
\leq \frac{\Delta}{\frac{1}{\cmin+\Lip}(\frac{1}{\cmin+\Lip}-\Delta)}.
\end{multline*}

\subsection{Proofs of lemmas}\label{app:proofs_lemmas}

\begin{proof}[Proof of \lemref{SW_contract}]
By \thmref{iso_contract}, we only need to prove that $\sw{\cdot}$
satisfies NUNA. Fix any $x\in\R^n$ and any index $k=1,\dots,n-1$.
Let $y=A_k x$. Note that we have
\[y_i = \begin{cases} x_i,&\text{ if $i<k$ or $i>k+1$,}\\
\frac{x_k+x_{k+1}}{2},&\text{ if $i=k$ or $i=k+1$.}\end{cases}\]
Take any indices $1\leq i\leq j\leq n$. We need to prove that $|\overline{y}_{i:j}|\cdot\psi(j-i+1)\leq \sw{x}$.
\begin{itemize}
\item Case 1: if $j<k$ or if $i> k+1$, then neither of the indices $k,k+1$ are included in the window $i:j$,
and therefore $x_{i:j}=y_{i:j}$ (i.e.~all entries
in the stretch of indices $i:j$ are equal). So,
\[|\overline{y}_{i:j}|\cdot\psi(j-i+1) = |\overline{x}_{i:j}|\cdot \psi(j-i+1)\leq\sw{x}.\]
\item Case 2:
If $i\leq k$ and $j\geq k+1$, then both indices $k,k+1$ are included in the window $i:j$. Since $y_k+y_{k+1} = x_k+x_{k+1}$
and all other entries of $x$ and $y$ coincide,
we can trivially see that
\[|\overline{y}_{i:j}|\cdot\psi(j-i+1) = |\overline{x}_{i:j}|\cdot \psi(j-i+1)\leq\sw{x}.\]
\item Case 3: if $i<k$ and $j=k$, then
\begin{align*}
&|\overline{y}_{i:j}|\cdot \psi(j-i+1)
=|\overline{y}_{i:k}|\cdot \psi(k-i+1)\\
&=\frac{\left|\sum_{\ell=i}^{k-1} x_\ell+\frac{x_k+x_{k+1}}{2}\right|\cdot \psi(k-i+1)}{k-i+1}\\
&=\frac{\left|\frac{1}{2}\sum_{\ell=i}^{k-1} x_{\ell} + \frac{1}{2}\sum_{\ell=i}^{k+1} x_{\ell}\right|\cdot \psi(k-i+1)}{k-i+1}\\
&\leq \frac{\left|\frac{1}{2}\sum_{\ell=i}^{k-1} x_{\ell}\right|\cdot \psi(k-i+1)}{k-i+1} + \frac{\left|\frac{1}{2}\sum_{\ell=i}^{k+1} x_{\ell}\right|\cdot \psi(k-i+1)}{k-i+1}\\
&= \frac{\psi(k-i+1)}{k-i+1} \cdot \bigg (\frac{1}{2} |\overline{x}_{i:(k-1)}|\cdot \psi(k-i) \cdot \frac{k-i}{\psi(k-i)} \\ & + \frac{1}{2} |\overline{x}_{i:(k+1)}|\cdot \psi(k-i+2) \cdot \frac{k-i+2}{\psi(k-i+2)} \bigg) \\
&\leq \sw{x}\cdot \frac{1}{2}\left[\frac{k-i}{\psi(k-i)}+ \frac{k-i+2}{\psi(k-i+2)}\right]\cdot \frac{\psi(k-i+1)}{k-i+1}\\
&\leq\sw{x} \cdot \frac{k-i+1}{\psi(k-i+1)} \cdot \frac{\psi(k-i+1)}{k-i+1} = \sw{x},
\end{align*}
where the last inequality holds since $i\mapsto i/\psi(i)$ is concave by assumption on $\psi$.
\item Case 4: if $i=k+1$ and $j>k+1$, by symmetry this case is analogous to Case 3.
\item Case 5: if $i=j=k$, then
\begin{equation*}
\begin{split}
|\overline{y}_{i:j}|\cdot \psi(j-i+1) & = |y_k|\cdot \psi(1) = |x_k + x_{k+1}| \cdot \psi(1)/2 \\ & = |\overline{x}_{k:(k+1)}|\cdot \psi(1)\leq |\overline{x}_{k:(k+1)}| \cdot\psi(2)\leq\sw{x},
\end{split}
\end{equation*}
since $\psi(1)\leq \psi(2)$ due to the assumption that $\psi$ is nondecreasing.
\item Case 6: if $i=j=k+1$, then by symmetry this
case is analogous to Case 5.
\end{itemize}
Therefore, $|\overline{y}_{i:j}|\cdot \psi(j-i+1)\leq \sw{x}$ for all indices $1\leq i\leq j\leq n$,
and so $\sw{y}\leq\sw{x}$, as desired.
\end{proof}

\begin{proof}[Proof of \lemref{SW_subg}]
For any indices $1\leq i\leq j\leq n$,
\[\overline{y}_{i:j} - \overline{x}_{i:j} = \sigma \overline{\eps}_{i:j},\]
and we know that $\sqrt{j-i+1}\cdot \overline{\eps}_{i:j}$ is subgaussian, that is,
\[\PP{\sqrt{j-i+1}\cdot\big|\overline{\eps}_{i:j}\big|>t}\leq 2e^{-t^2/2}\]
for any $t\geq 0$. Now we set $t=\sqrt{2\log\left(\frac{n^2+n}{\delta}\right)}$, and take a union bound over all $n + {n\choose 2}= \frac{n^2+n}{2}$ possible pairs 
of indices $i\leq j$. We then have
\[\PP{\max_{1\leq i\leq j\leq n}\sqrt{j-i+1}\cdot\big|\overline{\eps}_{i:j}\big|\leq \sqrt{2\log\left(\frac{n^2+n}{\delta}\right)}}\geq 1-\delta.\]
Setting $\psi(t) = \sqrt{t}$ proves that, on this event, $\sw{x-y}\leq \sigma\sqrt{2\log\left(\frac{n^2+n}{\delta}\right)}$, as desired.
For the bound in expectation, we have a similar calculation: it is known that $\EE{\max_{k=1,\dots,N} |Z_k|}\leq \sqrt{2\log(2N)}$ and $\EE{\max_{k=1,\dots,N} |Z_k|^2}\leq 8\log(2N)$  for any (not necessarily
independent) subgaussian random variables $Z_k$. Setting $Z_k = \sqrt{j-i+1}\cdot \overline{\eps}_{i:j}$ for each of the $N=\frac{n^2+n}{2}$ possible pairs $i,j$, we obtain
\[\EE{\max_{1\leq i\leq j\leq n}\sqrt{j-i+1}\cdot\big|\overline{\eps}_{i:j}\big|}\leq \sqrt{2\log(n^2+n)}\]
and
\[\EE{\Big(\max_{1\leq i\leq j\leq n}\sqrt{j-i+1}\cdot\big|\overline{\eps}_{i:j}\big|\Big)^2}\leq 8\log(n^2+n).\]
\end{proof}

\begin{proof}[Proof of \lemref{eiso}]
Assume that $x$ is $\eiso$-monotone, and fix any index $1\leq i \leq n$.
Let $j = \max\{k\leq n: \iso(x)_k = \iso(x)_i\}$. Then $i\leq j\leq n$, $\iso(x)_i = \iso(x)_j$,  and either $j=n$ or $\iso(x)_j<\iso(x)_{j+1}$. Therefore,
 we must have $x_j \leq \iso(x)_j$ by properties of the isotonic projection. (This is because, if $x_j>\iso(x)_j$, then writing $\ee{j}$ for the $j$th
 basis vector and taking some sufficiently small $\eps>0$, the vector $\iso(x) + \eps\cdot\ee{j}$ is an isotonic vector that is strictly closer
 to $x$ than $\iso(x)$, which is a contradiction.)
Therefore, $x_i \leq x_j + \eiso \leq \iso(x)_j + \eiso =\iso(x)_i + \eiso$. The reverse bound is proved similarly.

Now we turn to the converse. For any $1\leq i\leq j\leq n$, we have $x_i \leq \iso(x)_i + \eps \leq \iso(x)_j + \eps \leq x_j + 2\eps$,
where the first and third inequalities use the bound $\norm{x-\iso(x)}_{\infty}\leq \eps$, while the second uses the fact
that $\iso(x)$ is monotone.
\end{proof}

\begin{proof}[Proof of \lemref{slow_pava}]
For $i=1,\dots,n-1$, let $\ck_i = \{x\in\R^n: x_i\leq x_{i+1}\}$, which is a closed convex cone in
$\R^n$. We have $\ck_{\iso}  = \cap_{i=1}^{n-1} \ck_i $
and it's easy to see that $\iso_i(x) = \pr{\ck_i}(x)$. Hence the
slow projection algorithm defined in~\eqnref{slow_pava}
is actually a cyclic projection algorithm, that is, the iterates are given by
\[x^{(1)} = \pr{\ck_1}(x^{(0)}), \ \  x^{(2)} = \pr{\ck_2}(x^{(1)}),\ \  \dots,  \ \  x^{(n)} = \pr{\ck_1}(x^{(n-1)}),\ \  \dots.\]
In general, it is known that
a cyclic projection algorithm starting at some point $x=x^{(0)}$ is guaranteed to converge to some point in the intersection of the respective convex
sets, i.e.~$\lim_{t\rightarrow \infty} x^{(t)} = x^*\in\cap_{i=1}^{n-1}\ck_i = \ck_\iso$, but without any assumptions on the nature of the convex
sets $\ck_i$, this point may not  necessarily be the projection of $x$ onto the intersection of the sets
(\citet{bregman1965method,han1988successive}). Therefore, we need to check that for our specific choice of the sets $\ck_i$,
the cyclic projection algorithm~\eqnref{slow_pava}
in fact converges to $\iso(x)$ as claimed in the lemma.

We first claim that
\begin{equation}\label{eqn:iso_iter}\iso\big(\iso_i(x)\big) = \iso(x)\end{equation}
for all $x\in\R^n$ and all $i=1,\dots, n-1$. Assume for now that this is true.
Since $\iso(\cdot)$ is contractive with respect to the $\ell_2$ norm, the convergence $x^{(t)}\rightarrow x^*$ implies
that $\iso(x^{(t)})\rightarrow \iso(x^*)$.
Applying~\eqnref{iso_iter} inductively, we know that $\iso(x^{(t)}) = \iso(x^{(0)})=\iso(x)$ for all $t\geq 1$. On the other
hand, since $x^*\in\ck_{\iso}$, this means that $x^* = \iso(x^*)$. Combining everything, then, we obtain
\[\lim_{t\rightarrow \infty}x^{(t)} = x^* = \iso(x^*) = \lim_{t\rightarrow \infty}\iso(x^{(t)}) = \iso(x).\]

Finally, we need to prove~\eqnref{iso_iter}. Fix any index $i$ and any $x\in\R^n$.
If $x_i\leq x_{i+1}$, then $\iso_i(x)=x$ and the statement
holds trivially. If not, then $x_i>x_{i+1}$ and we have $\iso_i(x) = A_i x$ (recalling the definition
of $A_i$ in~\eqnref{Ai} earlier).
Now let $y=\iso(x)$ and $z=\iso(A_i x)$.
It is trivially true that, since $x_i>x_{i+1}$, we must have $y_i =y_{i+1}$. Also, $\inner{x-y}{z-y}\leq 0$ by properties
of projection to the convex set $\ck_\iso$, so we can calculate
\begin{equation}\label{eqn:innerprod}
\begin{split}
\inner{A_i x - y}{z-y} & = \inner{x-y}{z-y}  + \frac{x_{i+1}-x_i}{2}\cdot (z_i - y_i - z_{i+1}+y_{i+1})
 \\ & \leq \frac{x_{i+1}-x_i}{2}\cdot (z_i - z_{i+1}) \leq 0,
 \end{split}
 \end{equation}
where the last step holds since $z_i\leq z_{i+1}$ due to $z\in\ck_\iso$ and $x_i\geq x_{i+1}$ by assumption.
We also have $\norm{A_i x - z}^2_2\leq \norm{A_i x - y}^2_2$ since $z=\iso(A_i x)$, which combined
with~\eqnref{innerprod} proves that $y=z$.
Thus~\eqnref{iso_iter} holds, as desired.
\end{proof}

\begin{proof}[Proof of \lemref{spacing}]
Let $U_1,\dots,U_n\iidsim \textnormal{Unif}[0,1]$, and let $G(t) =\int_{s=0}^t g(s)\;\mathsf{d}s$
be the cumulative distribution function for the density $g$. Since $g\geq \cmin > 0$, $G:[0,1]\rightarrow[0,1]$ is strictly
increasing, and is therefore invertible. It is known that setting $Z_{(i)}=G^{-1}(U_{(i)})$ recovers
the desired distribution for the ordered sample points $Z_{(1)}\leq\dots\leq Z_{(n)}$.

Next, by \lemref{order_unif} below, with probability at least $1-\delta$, for all indices $0\leq i<j\leq n$,
\begin{equation}\label{eqn:Ubd1}\left|U_{(i)}-U_{(j)} - \frac{i-j}{n}\right| \leq \frac{\sqrt{3|i-j|\log((n^2+n)/\delta)} +2\log((n^2+n)/\delta) }{n}.\end{equation}
From this point on, assume that this bound holds. In particular, by taking $i=0$, this implies that
\begin{equation}\label{eqn:Ubd2}\left|U_{(j)} - \frac{j}{n}\right| \leq \frac{\sqrt{3j \log((n^2+n)/\delta)} +2\log((n^2+n)/\delta) }{n} \leq 4\sqrt{\frac{\log((n^2+n)/\delta)}{n}},\end{equation}
for all $j=1,\dots,n$,
by assuming that $ \log((n^2+n)/\delta)\leq n$ (if not, then this bound holds trivially since $U_{(j)}$ and $j/n$ both
lie in $[0,1]$).

Then, since $g$ is $\Lip$-Lipschitz, for $1\leq i< j\leq n$ we compute
\begin{multline*}
\left|\big(Z_{(i)} - Z_{(j)}\big) -\left( G^{-1}(i/n) -  G^{-1}(j/n)\right)\right| \\
=\left|\big(G^{-1}(U_{(i)}) - G^{-1}(U_{(j)})\big) -\left( G^{-1}(i/n) -  G^{-1}(j/n)\right)\right|\\
=\left|(U_{(i)} - U_{(j)})\cdot (G^{-1})'\left(sU_{(i)} + (1-s)U_{(j)}\right) -  \frac{i-j}{n}\cdot (G^{-1})' \left(\frac{si + (1-s)j}{n}\right)\right|,
\end{multline*}
where the last step holds by Taylor's theorem applied to the function 
\[s\mapsto \left(G^{-1}\left(sU_{(i)}+(1-s)U_{(j)}\right) - G^{-1}\left(\frac{si + (1-s)j}{n}\right)\right).\]
We can rewrite this as 
\begin{multline*}
\left|\big(Z_{(i)} - Z_{(j)}\big) -\left( G^{-1}(i/n) -  G^{-1}(j/n)\right)\right| \\
\leq\left|(U_{(i)} - U_{(j)}) - \frac{i-j}{n}\right|\cdot (G^{-1})'\left(sU_{(i)} + (1-s)U_{(j)}\right) +\\ \frac{j-i}{n}\cdot \left|(G^{-1})'\left(sU_{(i)} + (1-s)U_{(j)}\right)  -  (G^{-1})' \left(\frac{si + (1-s)j}{n}\right)\right|,
\end{multline*}
Since $G^{-1}$ has bounded first and second derivatives as in~\eqnref{Gderivs}, we then have
\begin{multline*}
\left|\big(Z_{(i)} - Z_{(j)}\big) -\left( G^{-1}(i/n) -  G^{-1}(j/n)\right)\right| \\
\leq\left|(U_{(i)} - U_{(j)}) - \frac{i-j}{n}\right|\cdot \frac{1}{\cmin}+ \frac{j-i}{n}\cdot \frac{\Lip}{\cmin^3}\cdot \left|\left(sU_{(i)} + (1-s)U_{(j)}\right)-\frac{si + (1-s)j}{n}\right|\\
\leq\frac{\sqrt{3(j-i)\log((n^2+n)/\delta)}  +2\log((n^2+n)/\delta)}{n}\cdot \frac{1}{\cmin}+ \frac{j-i}{n}\cdot \frac{4\Lip}{\cmin^3}\cdot \sqrt{\frac{\log((n^2+n)/\delta)}{n}},
\end{multline*}
applying the bounds obtained above in~\eqnref{Ubd1} and~\eqnref{Ubd2}. 
This proves the bound~\eqnref{spacing2} in the lemma. To prove the simpler bound~\eqnref{spacing1},
we calculate
\[\left|Z_{(i)} - G^{-1}(i/n)\right| = \left|G^{-1}(U_{(i)}) - G^{-1}(i/n)\right|\leq \frac{1}{c}\left|U_{(i)} - i /n\right| \leq \frac{4}{c}\sqrt{\frac{\log((n^2+n)/\delta)}{n}},\]
since $G^{-1}$ is $(1/c)$-Lipschitz and we can apply~\eqnref{Ubd2}.
\end{proof}

\begin{lemma}\label{lem:unif1}
Let $U_{(1)}\leq\dots\leq U_{(n)}$ be the order statistics of $U_1,\dots,U_n\iidsim\textnormal{Unif}[0,1]$.
For any  $\delta>0$,
\[\PP{\left| U_{(i)} - \frac{i}{n}\right| \leq \frac{\sqrt{3i\log(2n/\delta)} + 2\log(2n/\delta)}{n}\text{ for all $i=1,\dots,n$}}\geq 1-\delta.\]
\end{lemma}
\begin{proof}[Proof of \lemref{unif1}] 
Fix any index $i$. 
If $i<3\log(2n/\delta)$, then 
\[\frac{i}{n} - \frac{\sqrt{3i\log(2n/\delta)}}{n} \leq 0\]
and so trivially we have $U_{(i)} \geq \frac{i}{n} - \frac{\sqrt{3i\log(2n/\delta)}}{n}$.
If instead $i\geq 3\log(2n/\delta)$, then
suppose that $U_{(i)}\leq \frac{i}{n} - \frac{\sqrt{3i\log(2n/\delta)}}{n}\eqqcolon p$. This means
that at least $i$ many of the $U_k$'s lie in the interval  $[0,p]$.
Then
\begin{multline*}
\PP{U_{(i)}\leq \frac{i}{n} - \frac{\sqrt{3i\log(2n/\delta)}}{n}}
= \PP{U_{(i)}\leq p} = \PP{\textnormal{Binomial}(n,p)\geq i}\\
=\PP{\textnormal{Binomial}(n,p)\geq np\cdot\left(1 + \frac{\sqrt{3i\log(2n/\delta)}}{i-\sqrt{3i\log(2n/\delta)}}\right)}\\
\leq \exp\left\{-\frac{1}{3} np \left(\frac{\sqrt{3i\log(2n/\delta)}}{i-\sqrt{3i\log(2n/\delta)}}\right)^2\right\}
= \exp\left\{-\frac{1}{3}  \frac{\left(\sqrt{3i\log(2n/\delta)}\right)^2}{i-\sqrt{3i\log(2n/\delta)}}\right\} \leq \frac{\delta}{2n},
\end{multline*}
where  the  inequality uses the multiplicative Chernoff bound. 
Next, suppose that  instead we have
\[U_{(i)} \geq \frac{i}{n} + \frac{\sqrt{3i\log(2n/\delta)} + 2\log(2n/\delta)}{n} \eqqcolon p'.\]
This means that at most $i-1$ of the $U_k$'s lie  in the interval $[0,p']$. 
Then
\begin{eqnarray*}
& & \PP{U_{(i)}\geq \frac{i}{n} +\frac{\sqrt{3i\log(2n/\delta)}+ 2\log(2n/\delta)}{n}} =  \PP{U_{(i)}\geq p'}  \\ 
 & \leq &\PP{\textnormal{Binomial}(n,p')\leq i}\\
 & = & \PP{\textnormal{Binomial}(n,p')\leq np'\cdot\left(1 - \frac{\sqrt{3i\log(2n/\delta)}+2\log(2n/\delta)}{i+\sqrt{3i\log(2n/\delta)}+2\log(2n/\delta)}\right)}\\
& \leq & \exp\left\{-\frac{1}{2} np' \left(\frac{\sqrt{3i\log(2n/\delta)}+2\log(2n/\delta)}{i+\sqrt{3i\log(2n/\delta)}+2\log(2n/\delta)}\right)^2\right\}\\
& = & \exp\left\{-\frac{1}{2}  \frac{\left(\sqrt{3i\log(2n/\delta)}+2\log(2n/\delta)\right)^2}{i+\sqrt{3i\log(2n/\delta)}+2\log(2n/\delta)}\right\} \leq \frac{\delta}{2n},
\end{eqnarray*}
where again the first inequality uses the multiplicative Chernoff bound. 
Combining these two calculations,
\[\PP{\left|U_{(i)} - \frac{i}{n}\right|\leq \frac{\sqrt{3i\log(2n/\delta)}+2\log(2n/\delta)}{n}}\geq 1- \delta/n.\]
Finally, taking a union bound over all $i$, we have proved the lemma.
\end{proof}

\begin{lemma}\label{lem:order_unif}
Let $U_{(1)}\leq\dots\leq U_{(n)}$ be the order statistics of $U_1,\dots,U_n\iidsim\textnormal{Unif}[0,1]$,
and let $U_{(0)}=0$.
For any $\delta>0$,
\begin{equation*}
\begin{split}
& \mathbb{P}\Bigg \{\left|U_{(i)}-U_{(j)} - \frac{i-j}{n}\right| \leq \frac{\sqrt{3|i-j|\log((n^2+n)/\delta)}  +2\log((n^2+n)/\delta)}{n} \\ &
\text{ for all $0\leq i< j\leq n$} \Bigg \}\geq 1-\delta.
\end{split}
\end{equation*}
\end{lemma}

\begin{proof}[Proof of \lemref{order_unif}]
First, it is known that $U_{(j)}-U_{(i)}\sim\textnormal{Beta}\big((j-i),(n+1)-(j-i)\big)$ for all $0\leq i<j\leq n$.
In particular, $U_{(j)}-U_{(i)}$ has the same distribution as $U_{(j-i)}$, and so by \lemref{unif1}  applied with $k=j-i$ and with $2\delta/(n^2+n)$ in place of $\delta/n$,
\[\PP{\left|U_{(j)} - U_{(i)} - \frac{j-i}{n}\right| > \frac{\sqrt{3|i-j|\log((n^2+n)/\delta) }+2\log((n^2+n)/\delta)}{n}}\leq \frac{2\delta}{n^2+n}.\]
Taking a union bound over all ${n+1\choose 2}= \frac{n^2+n}{2}$ pairs of indices $i,j$, then, we obtain the desired bound.
\end{proof}

\end{document}